\documentclass[12pt]{article}
\usepackage{amsfonts,amssymb,amsmath,amsthm, import}
\usepackage{graphicx}
\usepackage[all]{xy}
\usepackage{pstricks,psfrag,xmpmulti,amscd,color}

\newtheorem*{thmA}{Theorem A}

\newtheorem*{corA}{Corollary A}
\newtheorem*{thmB}{Theorem B}
\newtheorem*{thmC}{Theorem C}
\newtheorem*{thmD}{Theorem D}
\newtheorem*{corE}{Corollary D}

 \divide\oddsidemargin by 2
\newtheorem{thm}{Theorem}[section]

\newtheorem{cor}[thm]{Corollary}
\newtheorem{lem}[thm]{Lemma}
\newtheorem{prop}[thm]{Proposition}

\theoremstyle{definition}
\newtheorem{defn}[thm]{Definition}
\newtheorem{rem}[thm]{Remark}

\numberwithin{equation}{section}

\font\nt=cmr7

\def\note#1
{\marginpar
{\nt $\leftarrow$
\par
\hfuzz=20pt \hbadness=9000 \hyphenpenalty=-100 \exhyphenpenalty=-100
\pretolerance=-1 \tolerance=9999 \doublehyphendemerits=-100000
\finalhyphendemerits=-100000 \baselineskip=6pt
#1}\hfuzz=1pt}

\def\be{\begin{equation}}
\def\ee{\end{equation}}

\renewcommand{\epsilon}{\varepsilon}

\newcommand{\ra}{\rightarrow}

\def\sm{\smallsetminus}

\newcommand{\diam}{\operatorname{diam}}

\newcommand{\dist}{\operatorname{dist}}

\renewcommand{\AA}{{\cal A}}

\newcommand{\CC}{{\cal C}}

\newcommand{\KK}{{\cal K}}

\newcommand{\PP}{{\cal P}}
\newcommand{\QQ}{{\cal Q}}

\renewcommand{\SS}{{\cal S}}

\newcommand{\WW}{{\cal W}}

\newcommand{\C}{{\mathbb C}}

\newcommand{\D}{{\mathbb D}}

\newcommand{\N}{{\mathbb N}}

\newcommand{\R}{{\mathbb R}}

\newcommand{\Z}{{\mathbb Z}}

\renewcommand{\Im}{\operatorname{Im}}

\newcommand{\e}{\epsilon}
\newcommand{\ov}{\overline}
\newcommand{\fnkzk}{f^{n_k}(z_k)}

\renewcommand{\ra}{\rightarrow}
\renewcommand{\Re}{\text{Re }}
\newcommand{\Chat}{\hat{\C}}

\newcommand{\cprime}{{c'}}
\newcommand{\ctilde}{{\tilde{c}}}

\renewcommand{\Re}{\operatorname{Re }}
\renewcommand{\epsilon}{\varepsilon}

\renewcommand{\cprime}{{c'}}
\renewcommand{\ctilde}{{\tilde{c}}}

\newcommand{\Ynu}{Y^{(n+1)}}

\newcommand{\suno}{{\s_1}}
\newcommand{\sdue}{{\s_2}}
\newcommand{\szero}{{{\mathbf{s}}_0}}
\newcommand{\sn}{{{\mathbf{s}}_n}}

\newcommand{\Psinp}{{\Psi_{n+1}}}
\newcommand{\Psin}{{\Psi_{n}}}

\renewcommand{\sm}{{{\s}^{-}}}
\renewcommand{\sp}{{{\s}^{+}}}

\newcommand{\s}{{\mathbf{s}}}
\newcommand{\st}{{\tilde{\mathbf{s}}}}

\renewcommand{\H}{{\mathbb{H}}}

\newcommand{\tsc}{{t_{\s,c}}}
\newcommand{\gsc}{{g_{\s}^{c}}}
\newcommand{\gscp}{{g_{\s}^{\cprime}}}

\renewcommand{\phi}{\varphi}

\newcommand{\psin}{{\psi_n}}
\newcommand{\psinp}{{\psi_{n+1}}}

\newcommand{\fc}{f_c}
\newcommand{\fcp}{f_{\cprime}}

\newcommand{\PPfc}{{\PP(f_c)}}

\newcommand{\Pn}{{P^{(n)}}}
\newcommand{\Pnp}{{P^{(n+1)}}}

\newcommand{\Yl}{{{Y_\ell}}}
\newcommand{\Ylt}{{{Y_{\tilde{\ell}}}}}
\newcommand{\Ylp}{{{Y'_{\ell}}}}
\newcommand{\Yltp}{{{Y'_{\tilde{\ell}}}}}

\newcommand{\QC}{{\operatorname{QC}}}

\newcommand{\jk}{{j_k}}
\newcommand{\zk}{{z_k}}
\newcommand{\nk}{{n_k}}

\title{Expansivity properties and rigidity for non-recurrent exponential maps}
\author{\small Anna Miriam Benini \thanks{This work  was partly supported by a DGAPA-UNAM fellowship at IMATE, Cuernavaca, Mexico, as well as by the Institute for Mathematical Sciences in Stony Brook, US. This work was partially supported by the ERC grant HEVO Holomorphic Evolution Equations n. 277691.}
\\ 
\small Dipartimento di matematica \\
\small Universita' di Tor Vergata\\
\small V. della Ricerca Scientifica	\\   
\small Roma, Italy\\ 
\small {\tt ambenini$@$gmail.com} 
}  
\begin{document}

\maketitle
\begin{abstract}
\emph{We show that an exponential map $f_c(z)=e^z+c$ whose singular value $c$ is combinatorially non-recurrent and non-escaping is uniquely determined by its combinatorics, i.e. the pattern in which its periodic dynamic rays land together. We do this by constructing puzzles and parapuzzles in the exponential family.
We  also prove a theorem about hyperbolicity of the postsingular set in the case that the singular value is non-recurrent. Finally, we show that boundedness of the postsingular set implies combinatorial non-recurrence if $c$ is in the Julia set.}
\end{abstract}

 \section*{Introduction}
In this paper we study the family of exponential maps $f_c(z)=e^z+c$, with $c\in\C$.  This is a  one-parameter family of functions each of which has exactly one asymptotic value $c$ and no critical values, which  makes it the transcendental analogue of the one-parameter families of unicritical polynomials $P_c(z)=z^D+c$.  Moreover, exponential maps can be seen as analytic as well as dynamical limits of the families of unicritical polynomials $(1+\frac{z}{d})^d+c$ (see for example \cite{BD2}). This allows several combinatorial results for the dynamical and parameter plane for exponential maps to be be studied in analogy with corresponding results for unicritical polynomials (see \cite{RS3} and Section~\ref{Non-recurrence and combinatorial non-recurrence}).


  For an  exponential map, the  singular value is called  \emph{non-recurrent} if  the postsingular set 
 
 \begin{displaymath}\PP(f_c) =\overline{\underset{n>0}\bigcup\{f_c^n(c)\}}\end{displaymath}
 
\noindent  does not contain the singular value $c$ itself
 {Observe that, although it is a less common choice, we  define the postsingular set not to include the singular value.}    
 Whenever the postsingular set is bounded, it avoids a left-half plane hence, by forward invariance, it avoids a neighborhood of $c$; hence in this case the singular value is automatically non-recurrent. 
 On the other side, it might well be that the postsingular set is unbounded without the orbit of $c$ tending to infinity. {For example, in the case of a Siegel disk with diophantine  rotation number, the Siegel disk itself is unbounded by \cite{He}, and its boundary is accumulated by the orbit of $c$ by (\cite[Corollary 2.10]{RvS} (moreover, in this case $c$ is recurrent). Observe that our condition of non-recurrence is equivalent to the condition  that $\PP(\fc)\cap\H_M=\emptyset$ for some left half plane $\H_M:=\{z\in\C, \Re z<-M\}$, since by forward invariance this implies that $\PP(\fc)\cap D_{e^{-M}}(c)=\emptyset$.
 Observe also that, while in the polynomial case the critical values either escape to infinity or have bounded orbit, in  the exponential case the singular value could have an  unbounded orbit which does not converge to infinity (in fact, for many parameters the singular orbit  is dense in $\C$).
   
The first result in this paper is hyperbolicity of $\PP(\fc)$ in the case that $c$  is {in the Julia set and} non-recurrent. {By he above mentioned \cite[Corollary 2.10]{RvS}, if $c$ is non-recurrent there are no Siegel Disks, so we can assume this throughout the paper.}


A  forward invariant closed set $K$ is called \emph{hyperbolic} (with respect to the Euclidean metric) if there exist $\ov{k},\eta>1$ such that for any $k>\ov{k}$ and for any $z \in K$, $|(f^k)'(z)|>\eta$. Observe that, although the term 'hyperbolic' is normally used only for compact sets,  we  are allowing the set $K$ to be unbounded.
{
\begin{thmA}[Hyperbolic sets]
Let $f_c(z)=e^z+c$ be a non-recurrent exponential map, and  $K\subset J(f_c)$ be a forward invariant closed set not containing parabolic points and  such that $K\cap D_{e^{-M}}(c)=\emptyset$ for  some $M\in\R$. Then  $K$ is hyperbolic with respect to the Euclidean metric.
\end{thmA}

\begin{corA}[Hyperbolicity of the postsingular set]
Let $f_c(z)=e^z+c$ such that $c\in J(f_c)$ is non-recurrent. Then  $\PP(\fc)$ is hyperbolic with respect to the Euclidean metric.
\end{corA}
}

 The rest of the paper is devoted to problems related to {\emph{rigidity}}. 
As in the polynomial setting, it is a natural and relevant question to ask  whether there exist  conditions which can be checked  in the dynamical plane  under which  two maps in the exponential  family are conformally conjugate, and hence correspond to the same parameter up to translation by $2\pi i$.
In this paper we give a sufficient combinatorial condition for this to happen, under the hypothesis of a specific kind of non-recurrence. 

The set of \emph{escaping points} 
\[I(f_c):=\{z\in\C: |f_c^{n}(z)|\ra\infty\}\]
 can be naturally described as a union of injective curves, called  \emph{dynamic rays} or \emph{hairs}, labelled by sequences in $\Z^\N$ in such a way that the dynamics of $f_c$ on the rays  is conjugate to the dynamics of the shift map $\sigma$ on $\Z^\N$ (see  \cite{BD1}, \cite{SZ1} and Section \ref{Dynamic rays} in this paper for existence and properties of dynamic rays in the exponential family).
Informally, dynamic rays can be thought of as curves from $(0,\infty)$ to $\C$, which tend to infinity as the parameter $t\ra\infty$.  If the limit  $t\ra0$ exists, we say that the dynamic ray \emph{lands}; {unfortunately this is not always the case, see for example \cite{Ja} for examples of non-landing rays in the exponential family.}  A  dynamic ray is called \emph{periodic} or \emph{preperiodic} if it is a periodic or preperiodic set under the dynamic of $f_c$; it is shown in \cite{Re1} that periodic and preperiodic rays land unless their forward orbit contains the singular value.

Two periodic or preperiodic dynamic rays landing at the same point, together with their common endpoint, form a curve $\Gamma$ disconnecting the plane. Given two such dynamic rays, we say that two points are \emph{separated} (by $\Gamma$) if they belong to different connected components of $\C\setminus\Gamma$.

A fundamental  question is  under which conditions the combinatorial data describing which periodic and preperiodic rays land together completely encodes the actual dynamics, and hence determines  the position of the map in the parameter plane uniquely. This property is  usually referred to as \emph{rigidity}.
Stated in this terms, rigidity can be asked for non-escaping parameters which do not belong to the closure  of hyperbolic components (or equivalently, non-escaping parameters for which all periodic point are repelling). Two maps $ \fc, \fc$ are called \emph{combinatorially equivalent} if their periodic and preperiodic rays land together in the same pattern (see  Definition~\ref{Combinatorial equivalence def} in Section~\ref{Dynamic rays}).

{A positive answer to the question of rigidity would imply density of hyperbolic parameters in the exponential parameter plane. The analogous rigidity  conjecture in the  parameter spaces of quadratic  polynomials is equivalent to the famous  MLC conjecture, according to which the Mandelbrot set is locally connected (see \cite[Theorem 10]{RS3}), and again implies density of hyperbolicity.  See \cite{RS3} for a description of the  parallel between rigidity problems in exponential versus polynomial setting, and \cite{Be} for rigidity of Misiurewicz paramters in the exponential family.}

The part  of the paper devoted to rigidity problems is structured as follows. After introducing the necessary combinatorial background in exponential dynamics, we  define Yoccoz puzzle and parapuzzles in the exponential setting in Section~\ref{Puzzles}.  
We then restrict ourselves  to the class of  \emph{combinatorially non-recurrent} parameters: a parameter $c$ is \emph{combinatorially non-recurrent} if  there is a suitable  collection (see Section \ref{Puzzles}) of preperiodic rays  which separate the singular value from the  postsingular set. 
Under the assumption of combinatorial non-recurrence we prove that two combinatorially equivalent maps are quasiconformally conjugate.

\begin{thmB}
Let $c, c'$ be non-escaping parameters, and $f_c$ be combinatorially non-recurrent. If $f_\cprime$ is combinatorially equivalent to $f_c$, then $f_\cprime$ is quasiconformally conjugate to $\fc$.
\end{thmB}

In Section~\ref{Rigidity} we show actual rigidity for combinatorially non-recurrent parameters. For parameters with bounded postsingular set, this can be obtained as a corollary of Theorem B together with a result of \cite{RvS} on the absence of invariant line fields (see also \cite{MS}). Otherwise, the proof is more involved. The final result is the following:

\begin{thmC}
Let $c, c'$ be non-escaping parameters, and $f_c$ be combinatorially non-recurrent.
If $f_\cprime$ is combinatorially equivalent to $f_c$, then $\cprime=c$.
\end{thmC}

We conclude  by showing  that if $c$ is the landing point of a dynamic ray with non-recurrent address (see Section~\ref{Non-recurrence and combinatorial non-recurrence}), then the usual notion of recurrence implies the stronger notion of combinatorial non-recurrence. 
{\begin{thmD}
Let $\fc$ be a {non-escaping} exponential map such that $c$ is non-recurrent and is the landing point of a dynamic ray $g_\s$ such that the length of $g_{\sigma^{n}\s}(0,t)\ra 0 $ uniformly in $n$ as $t\ra0$. Then $c$ is combinatorially non-recurrent. 
 \end{thmD}}
Observe that hyperbolic and parabolic parameters are non-recurrent in the usual sense but not in the combinatorial sense.

 Using a   theorem from \cite{BL} about accessibility of the singular value in the case in which $\PP(f)$ is bounded and contained in $J(f)$, we have the following corollary:

 {\begin{corE}
Let $\fc$ be an exponential map such that $c\in J(\fc)$ and $\PP(\fc)$ is bounded. Then $c$ is combinatorially non-recurrent. 
 \end{corE}}

The proof of Theorem D is independent from the rest of the paper, and relies on one side on known rigidity results for polynomials, and on the other side on the combinatorial similarity between exponentials and unicritical polynomials.

\subsection*{Acknowledgements}
 I am gratefully indebted to my advisor Misha Lyubich for suggesting this problem and for fundamental  discussions about this topic. I am thankful to Arnaud Ch\'eritat, Genadi Levin, Peter Makienko, Carsten Petersen, Lasse Rempe, Dierk Schleicher and Anna Zdunik for related discussions; to the IMATE in Cuernavaca and to the IMS in Stony Brook for their warm hospitality;  and to the first referee for suggesting clarifications and improvements in the structure of the paper as well as posing thoughtful questions challenging further investigations. 

\subsection*{Notation and terminology}
The complex plane is $\C$, the open unit disk is $\D$ and the Riemann sphere is $\hat{\C}$.
 We denote by $D_r(z)$ a disk of radius $r$ centered at the point $z$. 

The Euclidean diameter of a set $U'$ is denoted by $\diam U$, while the Euclidean length of a curve $\gamma$ is denoted by $\ell(\gamma)$. If $U$ admits a normalized  hyperbolic metric, the hyperbolic diameter of a set $U'\subset U$ is denoted by $\diam_U(U')$, while the hyperbolic length of a curve  $\gamma\subset U$ is denoted by $\ell_U(\gamma)$.
 
The Julia set of an exponential map $f$ is denoted by $J(f)$, and its Fatou set by $F(f)$. A parameter $c$ is called \emph{non-escaping} if $ \fc^n(c)\nrightarrow\infty $.  It is called \emph{hyperbolic} if it has an attracting periodic orbit, \emph{parabolic} if it has an indifferent periodic orbit with rational multiplier, and \emph{Misiurewicz} if the orbit of $c$ is finite. It is called \emph{Siegel} (respectively \emph{Cremer}) if  it has an indifferent periodic orbit $\{z_i\}^q_{i=1}$ with irrational multiplier in a neighborhood of which $f^q_c$ is linearizable (respectively non-linearizable). It is called \emph{non-recurrent} if the singular value is non-recurrent. 
A maximal open set of parameters which are hyperbolic is called a \emph{hyperbolic component}. Parabolic,  Siegel and Cremer parameters are on the boundaries of hyperbolic 
components.


\section{Hyperbolicity of the postsingular set}\label{Hyperbolicity of the postsingular set}

This section is dedicated to prove Theorem A, which will be used in the proof of Theorem B. The proof of Theorem A itself will not be used later.

Let $K$ be a closed forward invariant set not intersecting  a small disk $D_{e^{-M}}(c)$. By forward invariance $K$ does not intersect the left half plane $\H_M:=\{z\in\C; \Re z\leq M\}$ either. Observe that points in $K$ can still have arbitrarily large imaginary and real part.

For a topological disk  $V$,  we call $\CC(n,V)$ is  the set of  connected components of $f^{-n}(V)$, and $\CC_K(n,V)$ be the set of  connected components of $f^{-n}(V)$ which intersect $K$. By forward invariance of $K$, if $U\in \CC_K(n,V)$, $f^j(U)\in \CC_K(n-j,V)$ for any $j\leq n$. Similarly, if $V'\subset V$, and $U'\in \CC_K(n, V')$, then $U'\subset U$ for some $U\in\CC_K(n,V)$. Observe also that if $c$ is non-recurrent, the postsingular set $\PP(f)$ satisfies the hypothesis on $K$.

We recall that a family of univalent functions $\{\phi_k\}$, $\phi_k: V\ra \C$, with $V$ simply connected, is \emph{normal} if every sequence either has a convergent subsequence, or escapes any compact set (see \cite{Mi}). By Montel's Theorem, any family omitting three values is normal. So, for any    simply connected neighborhood $V$ of a point $z\in\C$ which omits a periodic orbit of period $>3$, the family of univalent inverse branches of $f$ defined on $V$ is normal.

The main tool to prove Theorem A is the following proposition (compare with Theorem 1.1 in \cite{ST}, \cite[Theorem 2.7]{RvS}, and \cite[Proposition 3]{Ly1}):

\begin{prop}[Local expansivity]\label{Local expansivity}
Let $f_c$ be a non-recurrent exponential map. Let $K\subset J(f_c)$ be a forward invariant closed set not intersecting $D_{e^{-M}}(c)$ for some $M>0$ and not containing parabolic points. Then for any $\epsilon>0$ and any $z_0\in K$ there exists  $\delta>0$ such that for any  $U\in \CC_K(n,D_{\delta}(z_0))$ the following two statements hold:
\begin{itemize}
\item[a.] $\diam U<\epsilon$ and $f^n: U\rightarrow D_{\delta}(z_0)$ is univalent.
\item[b.]For all $\epsilon'>0$ there exists $n_{\epsilon'}$ such that if $n>n_{\epsilon'}$ and $U\in \CC_K(n,D_{\delta'(z)})$ with $\delta'<\delta$, then $\diam U<{\epsilon'}$. 
\end{itemize}
\end{prop}
To prove Proposition~\ref{Local expansivity} we need  
 two lemmas. The first one is  \cite[Lemma 2.1]{ST}.

\begin{lem}
\label{Shishitan}
 For any $0<\delta<1$, there exists a constant $C(\delta)$ such that for any univalent map $g: U\ra \D$ with $U$ simply connected, and for any connected component $U'$ of $g^{-1}(D_\delta(0))$, $\diam_U(U')\leq C(\delta)$. Moreover $\lim_{\delta\ra0}C(\delta)=0$. 
\end{lem}

The next lemma is a bit technical but it is needed to deal with   sequences of pullbacks which go to infinity, in the sense that they do not have any finite accumulation point.

\begin{lem}
\label{Mongolfiera} Let $f(z)=e^z+c$, $c\in J(f)$ non-recurrent and $K$ be as in Proposition~\ref{Local expansivity}. 
Let $z_0\in\C$, $V=D_\delta(z_0)$ with $\delta$ small. Suppose that  there exists a sequence of univalent pullbacks $U_k\in \CC_K(n_k, V)$, {$k\ra\infty$,} with $\diam U_k\geq\epsilon_k> \delta/2^{n_k}$ for $k$ sufficiently large. Then there exists a sequence of integers $j_k<n_k$ such that $U_k'\in\CC_K(n_k-j_k-1, V)$ has a finite accumulation point and $\diam U_k'\geq\epsilon'_k> e^{-M}(1-e^{-\epsilon_k})$. In particular, if the ${\epsilon_k}$ are bounded away from $0$ so are the $\epsilon_k'$, and if $\epsilon_k\ra\infty$, $\epsilon'_k\ra e^{-M}$.
 Moreover, if $\inf\epsilon_k>0$, and $n_k\ra\infty$, then  $n_k-j_k\ra\infty$.
\end{lem}

\begin{proof} Let $S=\{x, -M<\Re x<\log2\}$ with $M$ as in Proposition~\ref{Local expansivity}. Observe that $|f'(x)|\geq2$ whenever $\Re x\geq \log 2$,  so, as $\diam U_k> \delta/2^{n_k}$, there exists a minimal $j_k\geq0$ such that  $f^{j_k}(U_k)\cap \{z\in\C;\Re z<\log2\}\neq\emptyset$; since  $K\cap U_k\neq \emptyset$ for all $k$ and $K$ is forward invariant,   $f^{j_k}(U_k)\cap S\neq\emptyset$.

 Let $U_k':=f^{j_k+1}(U_k)$. As $f^{j_k}(U_k)\cap S\neq\emptyset$, the sets  $U_k'$ have a finite accumulation point in the annulus $A=f(S)$, centered at $c$, with internal radius $e^{-M} $ and external radius $2$. It is only left to estimate $\diam U_k'$. 
As $j_k$ is minimal, $\diam f^{j_k}(U_k)\geq 2^{j_k}\diam U_k\geq \epsilon_k$. Between all connected sets of diameter $L$ which intersect $S$ and on which $f$ is univalent, $|f'|$ is smallest along a horizontal segment of length $L$ going to the left of $S$,  whose image is a segment of length $e^{-M}-e^{-L-M}$, so 
$\diam U_k'\geq e^{-M}(1-e^{-\epsilon_k})$.

To prove that under the given hypothesis $n_k-j_k\ra\infty$ it is enough to show that the sequence $\{j_k\}$ is bounded. For each $k$, by minimality of $j_k$ we have that $\diam f^{j_k}(U_k)\geq  2^{j_k}\diam U_k$; hence if $\inf\diam U_k>0$ and $j_k\ra\infty$, $\diam f^{j_k}(U_k)\ra\infty$ (and by definition of $j_k$,  $f^{j_k}(U_k)\cap S\neq\emptyset$). On the other side, up to considering a translate of  $f^{j_k}(U_k)$ by $2\pi i q$ for some interger $q$ (which is still a $n_k-j_k$th preimage of $V$ since $f$ is $2\pi i $ periodic) we obtain that the sets in question all intersect the compact set $S\cap\{z\in\C;|\Im z|<2\pi\}$ and have diameter tending to infinity, contradicting normality of inverse branches on $V$.
\end{proof}
\begin{proof}[Proof of Proposition~\ref{Local expansivity}]
\begin{itemize}
\item[a.]
 By Theorem 2.7 in \cite{RvS}\footnote{For the exponential family, non-recurrence implies strong non-recurrence as defined in \cite{RvS}. The proof  in \cite{RvS} works in the same way by substituting pullbacks  along the postsingular set with pullbacks along $K$. Unfortunately the rest of Proposition~\ref{Local expansivity} can  not be as explicitly deduced from results in \cite{RvS}.} every non-parabolic point is regular, so by Theorem 2.5 in \cite{RvS} there exists $\delta_1$ such that for any $U$ as in the claim $f^n:U \ra D_{\delta_1}$ is univalent.
Let us first assume that there exists $R>0$ such that $\diam U<R$ for any $U\in\CC_K(n, D_{\delta_1}(z_0))$. Then for any such $U$, $U\subset D_R(z)$ for some $z\in U$. On $D_R(z)$, the euclidean density is bounded by $R$ times the hyperbolic density on $D_R(z)$. 
Let $\delta$ be  sufficiently small so that  $C(\delta/\delta_1)<\epsilon/R$ as given by  Lemma~\ref{Shishitan}. 
Then using {the comparison principle for } the hyperbolic metric, we get that for any $U'\subset U\subset D_R(z)$ with $U'\in\CC_K(n, D_\delta(z_0))$,
\[ \diam U'\leq R\diam_{D_R(z)} (U')\leq R \diam_{U} U'\leq \epsilon.\]

 We now show that there exists $R>0$ such that $\diam U<R$ for any $U\in\CC_K(n, D_{\delta_1}(z_0))$. Let $A$ be the annulus from the proof of   Lemma~\ref{Mongolfiera}. By normality of inverse branches defined on $D_\delta(z_0)$, there cannot be a sequence $U_k$ intersecting $A$ with $\diam U_k\ra\infty$, so there is some $R$ such that $\diam U<R$ for all $U$ with $U\cap A\neq\emptyset$ and  the claim holds for all $U$ with $U\cap A\neq\emptyset$. It is only left to show that there is some $R'$ such that $\diam U<R'$ for all $U\in \CC_K(n_k, D_{\delta}(z))$, not necessarily intersecting $A$. Up to making  $\delta_1$ smaller we can assume that  $\diam U<e^{-M}/2$  for all $U\in \CC_K(n, D_{\delta_1}(z_0))$  intersecting $A$. Suppose by contradiction that there exists a sequence  $U_k\in \CC_K(n_k, D_{\delta_1}(z) )$ with $\diam U_k\ra\infty$. By Lemma~\ref{Mongolfiera}, this gives  a sequence of $U_{k}'$ intersecting $A$ with diameter tending to $e^{-M}$, contradicting the choice of $\delta_1$. The claim then holds up to making $\delta$ smaller so that $C(\delta/\delta_1)<\min(\epsilon/R',\epsilon/R)$.
\item[b.]
  If not, there exist $\epsilon>0$ and a sequence $U_k\in\CC_K(n_k, D_{\delta}(z))$,  $n_k\ra\infty$, with $\diam U_k\geq \epsilon$. As $n_k\ra\infty$, eventually $\epsilon/{2^{n_k}}<\delta$.
 Thanks to Lemma~\ref{Mongolfiera}, up to replacing $U_k$ with $U_k'$, we can assume that the $U_k$ have a finite accumulation point $y$. 
Let $\{\phi_k\}$ be   the sequence of inverse branches defined on $D_{\delta_1}\supset \ov{D_{\delta}}$ such that $\phi_k: D_{\delta_1}\ra U_k$. 
By   normality, up to passing to a subsequence the $\phi_k$ converge uniformly on compact subsets of $D_{\delta_1}$ to a limit function $\phi$, which is non-constant because $\diam U_k\geq \epsilon$ and $\ov{D_{\delta}}\subset D_{\delta_1}$. It follows that there is a neighborhood of $\phi(z)$ which is mapped inside $D_\delta(z)$ under infinitely many iterates of $f$, contradicting $\phi(z) \cap J(f)\neq\emptyset$.
\end{itemize}  

\end{proof}

\begin{proof}[Proof of Theorem A] 
If $K$ is not hyperbolic, there are  $n_k\rightarrow \infty$, and $z_k\in K$, such that $|(f^{n_k})'(z_k)|\leq1$.
As $K$ is closed and forward invariant, any $\{f^{n_k}(z_k)\}$, as well as any finite accumulation point thereof,  belongs to $K$.
Assume first that   $\fnkzk$ has a finite accumulation point $y\in K$.
Fix $\epsilon>0$ and let $V=D_{\delta}(x)$ as given by Proposition~\ref{Local expansivity}.
Up to considering a subsequence, we can assume that, for large $k$, $\fnkzk\in V$.
    For each such $k$, let $U_k$ be the component of $f^{-n_k}(V)$ containing $z_k$. By  Proposition~\ref{Local expansivity} part a., $\diam (f^j(U_k))\leq\epsilon$ for $j=0\dots n_k$, and  $f^{n_k}:U_k\rightarrow V$ is univalent so it has a  local inverse $\phi_k:V\rightarrow U_k$.     
  The family $\{\phi_k\}$ is normal because it is a family of univalent inverse functions defined in $V$ which omits at least three points; any limit function $\phi$ for $\phi_k$ is constant by  Proposition~\ref{Local expansivity} part b. This contradicts the initial assumption that 
\[|\phi'(y)|=\lim_{k\rightarrow\infty}|\phi_k'(f^{n_k}(z_k))|=\lim_{k\rightarrow\infty}\left|\frac{1}{(f^{n_k})'(z_k)}\right|\geq 1.\]
 
Now suppose  that  $\fnkzk\ra\infty$, let $A, S$ be the annulus  and the strip as in Lemma~\ref{Mongolfiera}. To reduce to the previous case it is enough to find a sequence  $j_k<n_k$ such that $j_k\ra\infty$, $f^{\jk+1}(\zk)\in A$, and $|({f^{\jk+1}})'(z_k)|\leq 1$. 
 
If $f^{\nk}(\zk)\in A$, let $j_k=n_k$. Otherwise, for any $k$ let $j_k\leq n_k-1$ be the greatest integer such that $f^{\jk}(z_k)\in A$. Observe that the condition $|(f^{n_k})'(z_k)|\leq1$, plus the fact that $f^{\nk}(\zk)\in K$, implies that such a $j_k$ exists, otherwise we would have $|(f^{n_k})'(z_k)|\geq 2^{n_k}>1$.
 
 To prove that $|({f^{\jk+1}})'(z_k)|\leq 1$, observe that 
 
\begin{align*}
1&\geq |\fnkzk'|=
\prod_{i=0}^\jk e^{\Re f^i(\zk)}\prod_{i=\jk+1}^{n_k-1} e^{\Re f^i(\zk)}\geq \\
&\geq \prod_{i=0}^\jk e^{\Re f^i(\zk)} 2^{\nk-\jk-1}
  \Rightarrow \prod_{i=0}^\jk e^{\Re f^i(\zk)}= |(f^{\jk+1})'(\zk)|\leq 1/2^{n_k-j_k-1}\leq 1.
\end{align*} 
Similarly, because $\Re z>-M$ for all $z\in K$,

\[
1\geq |\fnkzk'|\geq
\left(\prod_{i=0}^\jk e^{-M}\right) 2^{\nk-\jk-1}=e^{-M(j_k+1)}2^{n_k-j_k-1}\]

hence, as $\nk\ra\infty$,  $\jk\ra\infty$ as well. 

\end{proof}

{
\begin{cor}\label{zero area}
Let $K$ be as in Theorem A. Then $K$ has Euclidean  measure zero. 
\end{cor}

\begin{proof} Let $\Omega:=\C\setminus  D_{e^{-M/2}}(c) $ and consider the hyperbolic metric in $\Omega$. The set  $K$ is hyperbolic with respect to the hyperbolic metric in $\Omega$, hence since it is forward invariant its hyperbolic area is either $0$ or $\infty$. Since the hyperbolic area of $\C\setminus  D_{e^{-M}}(c) $ in $\Omega$  is finite, the hyperbolic area of  $K$ is $0$, hence its Euclidean area is also $0$.
\end{proof}
}
 
\section{Quasiconformal maps and holomorphic motions}\label{Qc maps}

Here we recall the definition and the main properties of quasiconformal maps and of holomorphic motions. These results and their proof can be found for example in  \cite{Ah}; they will be used in Sections~\ref{Puzzles},\ref{qc rigidity}, and \ref{Rigidity}.

Let  $\psi:\Chat\ra\Chat$ be a homeomorphism with partial derivatives $\partial_z\psi,\partial_{\ov{z}}\psi$ in the sense of distribution and which are locally $L_1$. We associate to $\psi$ its \emph{Beltrami coefficient} $\mu_\psi:=\frac{\partial_{\ov{z}}\psi}{\partial_{{z}}\psi}$. By definition $\mu_\psi(z)$ is a measurable function.
The map $\psi$ is called \emph{quasiconformal} if $\|\mu_\psi\|_\infty<1$.
The quantity $K:=\sup_{z}\frac{\mu_\psi(z)+1}{1-\mu_\psi(z)}$ is called the \emph{dilatation} of $\psi$; if  $\psi$ has dilatation $K$, it is called $K$-quasiconformal.
Inverses of $K$-quasiconformal maps are $K$-quasiconformal; moreover,  if $g$ is conformal and $\psi$ is $K$-quasiconformal, $g\circ \psi$ as well as  $\psi\circ g$ are $K$-quasiconformal. 

\begin{lem}[Weyl's Lemma] If $f$ is a quasiconformal map with $\mu_f=0$ outside a set of Lebesgue measure zero, then $f$ is conformal.
\end{lem}

We say that a sequence of functions $\{\psi_n\}$ has a limit function $\psi$ if there is a subsequence converging to $\psi $.  One of the remarkable properties of $K$-quasiconformal maps is their compactness.
 
\begin{lem}\label{Precompactness} Let  $\{\psi_n\}$ be a family of \
$K$-quasiconformal functions which coincide on at least three points. Then there exists at least one  limit function $\psi$, and it is $K$-quasiconformal.
\end{lem}

By definition, a  quasiconformal map $\psi$ induces a Beltrami coefficient with $ \|\mu_\psi\|_\infty<1$; it is a natural question to   ask whether  any allowable Beltrami coefficient (which can be thought of as  a  measurable function with modulus less than 1) corresponds to a quasiconformal map. The answer is provided  by the Measurable Riemann Mapping Theorem.

\begin{thm}[Measurable Riemann Mapping Theorem] Let $\{\mu_\lambda\}$, $\|\mu_\lambda\|_\infty<1$, be a family of Beltrami coefficients depending holomorphically on  a parameter $\lambda$. Then there is a family of quasiconformal maps $\psi_\lambda$, depending holomorphically on $\lambda$, such that $\mu_\lambda=\partial_z \psi_\lambda/\partial_{\ov{z}} \psi_\lambda$.
\end{thm}

A Beltrami coefficient defines a measurable field of ellipses in the tangent space, with bounded ratio between the length of the  major and minor axis. On the other side, a measurable field of ellipses with bounded ratio between the major and minor axis  defines a Beltrami coefficient. This field of ellipses defines a so-called \emph{conformal structure} on $\C$.
 For a holomorphic map $g$, $\mu_g=0$, and the induced field of ellipses is in fact a field of circles. We refer to this field of circles as the \emph{standard conformal structure} ${\sigma_0}$ on $\C$. 

Given a holomorphic function $f$, an \emph{invariant line field} is a field of lines defined in the tangent space which is invariant under $f$.
A Beltrami coefficient $\mu$ defines an invariant line field through the directions of the major axis of the induced field of  ellipses in the tangent space. The induced line field is invariant  if and only if $f^*\mu=\mu$.
The following theorem is a special case of Theorem 1.1 in \cite{RvS}.
\begin{thm}[Absence of invariant line fields]\label{Absence of line fields}
Let $f_c(z)=e^z+c$ such that $\PP(\fc)$ is bounded. Then  $J(\fc)$ supports no invariant line fields.
\end{thm}

Holomorphic motions are another widely used tool  in one-dimensional complex dynamics.
\begin{defn}
Let $(\Lambda,*)$ be a topological disk in $\C$ with a marked point $*$, and let $X$ be a subset of the Riemann sphere.
 A  \emph{ holomorphic motion} $\bf{h}$ of $X$ over $(\Lambda,*)$ is a family of injections $h_\lambda:X\ra\C$, $\lambda\in\Lambda$, depending holomorphically on $\lambda$ for each fixed $x\in X$, and such that $h_*$ is the identity. Define $X_\lambda:=h_\lambda(X).$ A set $X$ is said to \emph{move holomorphically} over $\Lambda$ if such a holomorphic motion exists. 
\end{defn}

{We will need the following result about extensions of holomorphic motions (see \cite[Theorem 1]{BeR}):   
\begin{thm}[\bf{Bers-Royden extension}]\label{Lambda Lemma}
 Let $\bf{h}$ be a holomorphic motion of a set $X\subset\C$ over  the disk $(\Lambda,*)$. Then each  $h_\lambda$ extends to a  quasiconformal self-map of $\hat{\C}$.
\end{thm}

}

%

\section{Combinatorics for   exponential maps}
In this section we recollect the relevant information needed about combinatorics in the exponential family.
In particular, we introduce dynamic and parameter rays and we describe the structure that they induce in the dynamical and parameter plane respectively. Both have been first introduced in \cite{BD1} with the term \emph{hairs}. A full classification in terms of rays of the set of escaping points and escaping parameters has subsequently been carried out  in \cite{SZ1} and \cite{FS} respectively. 

\subsection{Dynamic rays}\label{Dynamic rays}
 
 For all of this section, $f_c(z)=e^z+c$. 
Let $F(t)=e^t-1$ be a model function for real exponential growth, $\sigma$ be the left-sided shift map acting on $\Z^\N$. 

A sequence $\s=s_0 s_1\ldots\in\Z^\N$ is called \emph{exponentially bounded} if  there exists  $x\in\R$ such that    $ |s_i|\leq F^i (x)$ for every $i\in\N$. For the exponential family, the set of exponentially bounded sequences is called the set of  \emph{addresses}\footnote{Addresses are referred in \cite{SZ1} and \cite{FS} as \emph{exponentially bounded addresses} to stress out the fact that they satisfy the condition above.}
 and is denoted by $\SS$. An address is called \emph{periodic} or  \emph{preperiodic} if it is a periodic or preperiodic sequence.

 For any two sequences $\s=s_0 s_1 s_2\ldots$ and $\s'=s'_0 s'_1 s'_2
 \ldots$ in $\SS$, we consider the distance

 \begin{equation}\label{metric}
|\s-\s'|=\underset{s_i\neq s_i'}\sum{\frac{1}{2^i}}. \end{equation}

An address is called \emph{bounded} if $\|\s\|:=\sup_i |s_i|<\infty$.  We refer to $\SS$  as the \emph{combinatorial space} for $f$.  Observe that $\SS$ is endowed with the lexicographic order.

Definition, existence and properties of \emph{dynamic rays} for the exponential family are summarized in  the following theorem (\cite{SZ1}, Proposition 3.2 and Theorem 4.2): 

\begin{thm}[Dynamic rays, \cite{SZ1}]\label{ExistenceRays}
Let $c\in \C$ and $\s\in\SS$ be an address. Then, there exist $\tsc$ and a unique  maximal injective curve  $\gsc:(\tsc,\infty) \to I(f_c)$ such that
\begin{enumerate}
\item[(a)] $f_c(\gsc(t))=g^c_{\sigma\s}(F(t)). $
\item[(b)] $f_c^n (\gsc(t))=2\pi i s_n + F^n(t)+ o(e^{-F^n(t)})$ as $t\ra \infty$. 
\end{enumerate}
Moreover,
\begin{enumerate}
\item[(c)] For any $\s$, $\gsc(t)$ depends analytically on $c$ unless $c\in g^c_{\sigma^n\s}$ for some $n\in\N$. 
\end{enumerate}
\end{thm}

The curve $\gsc$ is called the \emph{dynamic ray} (or just \emph{ray}) of address $\s$. If $c$ is non-escaping, $\tsc$ does not depend on $c$.
We  omit the index $c$ and write $g_{\s}(t)$ when this creates no ambiguity. The estimates in $(b)$ together with the fact that dynamic rays do not intersect induce a \emph{vertical order} on dynamic rays near infinity, corresponding to the lexicographic order on $\SS$. Injectivity together with  $(c)$ imply  that the functions $\gsc\circ (g^{\ctilde}_\s)^{-1}$ define a holomorphic motion of the  dynamic ray $g^{\ctilde}_\s$ over any neighborhood of $\ctilde$ on which the ray $\gsc(t)$ is well defined.

A dynamic ray $g_\s^c$ is called \emph{periodic} (resp. \emph{preperiodic}) if $\s$ is a periodic sequence (resp. strictly preperiodic). 
A  dynamic ray $\gsc$ is said to  \emph{land} at $z\in \C$ if $\lim_{t\to \tsc}\gsc(t) = z$.
It is shown in \cite{Re1} that periodic and preperiodic dynamic rays land unless one of their forward images contains the singular value.  {If a ray $\gsc$ lands at a non-escaping point, then $\tsc=0$ (see \cite{FS} and the definition of minimal potential and on fast addresses  there).}

\begin{defn}[Combinatorial equivalence]\label{Combinatorial equivalence def}
 Given a set $\AA$ of addresses, we say that the rays with addresses in $\AA$ for $f_c,\fcp$  \emph{land together in the same pattern} whenever the following condition is satisfied: two rays  $g^c_\s,g^c_{\s'}$ with addresses in $\AA$ land together 
 if and only if  $g^{c'}_{\s},g^{c'}_{\s'}$ land together.
 If all periodic and preperiodic dynamic rays for $f_c$ and $\fcp$ land together in the same pattern,  we say that $f_c$ and $\fcp$ are \emph{combinatorially equivalent}. 
 \end{defn}
 
We have the following result about dependence  on the external address $\s$ (see e.g.  \cite[Lemma 4.7]{Re2}).

\begin{lem}[Transversal Continuity]\label{ContRays}
Let $\{\s_n\}\subset\SS$ be a sequence of addresses converging to an address $\s$ with $t_{\sn,c}\ra t_{\s,c}$.  Then $g_\sn(t)$ converges uniformly to $g_\s(t)$ on all intervals $[t_*,\infty]$ with $t_*>t_{\s,c}$.
\end{lem} 


The following result is  well known and its proof is provided for the reader's convenience.

\begin{lem}[Landing of preimages of rays]\label{Splitting} Let $f(z)=e^z+c$ be an exponential map, 
and $\suno< \sdue$ be two  addresses  such that $g_\suno$ and $g_\sdue$ land together at some point $z\neq c$.
 If $\ov{g_\suno\cup g_\sdue}$ separates $c$ from $-\infty$,  for any $k\in\Z$, $g_{k\suno}$ lands together with $g_{(k-1)\sdue}$. If not, 
for any $k\in\Z$, $g_{k\suno}$ lands together with $g_{k\sdue}$.
\end{lem}

\begin{proof}
As $f$ is a local homeomorphism near $z$, preimages of $g_\suno$ and $g_\sdue$ land together pairwise at preimages of $z$. 
If the curve $\gamma:=\ov{g_\suno\cup g_\sdue}$ has winding number one with respect to $c$,  the imaginary part increases by $2\pi $ along any connected component of $f^{-1}(\gamma)$. 
If the winding number is zero, there is no increase in the imaginary part.
By the asymptotic estimates in Theorem \ref{ExistenceRays}, any two rays in a connected  component of $f^{-1}(\gamma)$ have to differ by exactly one unit in the first entry of their addresses in the first case, and by none in the second case.
\end{proof}

\subsection{Parameter rays, parabolic wakes, Misiurewicz wakes}\label{Parameter rays, parabolic wakes, Misiurewicz wakes}

The set of escaping parameters also consists of curves tending to infinity, called \emph{parameter rays} (see \cite{FS}, Theorem 3.7).

\begin{thm}[Parameter rays]\label{Parameter rays} 
Let $\s\in\SS$. Then there is $t_\s>0$, and a unique maximal  injective curve $G_\s: (t_\s,\infty)\ra\C$, such that, for all $t > t_\s$, $c=G_\s(t)$ if and only if $c=\gsc(t)$.  Also, $|G_\s(t)-(t+2\pi i s_0)|\rightarrow 0$ as $t\rightarrow \infty.$ 
\end{thm}

Recall that a hyperbolic component $W$ is a maximal connected set of parameters with an attracting periodic  orbit of the same period. There is a unique hyperbolic component $W_0$ of period 1 (the period of a hyperbolic component $W$ is the period of the attracting periodic  orbit for parameters in $W$). Parameter rays are approximated from above and from below by other parameter rays and by curves in hyperbolic components; see \cite[Lemma 7.1]{BR}, as well as the proof of the Squeezing Lemma in \cite[Section 4]{RS2} and the proof that the exponential bifurcation locus is not locally connected in \cite[Theorem 5]{RS3}.

It is known that parameter rays with periodic and preperiodic addresses land at parabolic and Misiurewicz parameters respectively (see \cite{Sc0}, \cite[Theorem 8.5]{RS2}, \cite{BeRe}).

The landing pattern of parameter rays with periodic and preperiodic addresses carves the structure of parameter plane, and is related to the landing pattern of dynamic rays in dynamical plane, as well as the position of hyperbolic components. This is exemplified in the next two theorems. 
The first theorem describes the relation between the landing pattern of parameter and dynamic  rays with periodic address (see  Proposition 4 and Proposition 5 in \cite{Re1}, as well as \cite{RS1}).

\begin{prop}[Parabolic wakes]\label{Parabolic wakes}
Given a hyperbolic component $W$, there are exactly two parameter rays $G_\sp, G_\sm$ of periodic addresses $\sp$ and $\sm$ which land together  on $\partial W$ and separate $W$ from $-\infty$. The connected component   $\WW(W)$ of $\C\setminus \ov{G_\sp\cup G_\sm}$  containing $W$ is called the \emph{(parabolic) wake} of $W$. The  dynamic rays $g_\sp$, $g_\sm$ and any of their forward iterates move holomorphically over  $\WW(W)$, and they land together in the dynamical plane for $f_c$ if and only if $c$ belongs to $\WW(W)$.  The rays $g_\sp$ and $g_\sm$ are called \emph{characteristic dynamic rays}, while  $G_\sp, G_\sm$  are called \emph{characteristic parameter rays}.
\end{prop}

If $W_1, W_2$ are  two  hyperbolic components such that  $\partial W_1\cap \partial W_2\neq\emptyset$, and the period of $W_2$ is higher than the period of $W_1$, we say that  we say that $W_2$ is \emph{attached} to $W_1$.

The situation for preperiodic addresses is less explicitly stated in the literature. However, it is shown in \cite{LSV} (Theorem 3.4 and Corollary 3.5) that parameter rays with preperiodic addresses land at Misiurewicz parameters, and that each Misiurewicz parameter $c_0$ is the landing point of finitely many parameter rays $G_{\s_1}\ldots G_{\s_q}$ with preperiodic addresses. Moreover, the dynamical rays of addresses $\s_1\ldots \s_q$ land at $c_0$ in the dynamical plane for $f_{c_0}$. It is also not hard to prove that any Misiurewicz parameter is contained in a parabolic wake attached to $W_0$, by using the fact that all repelling fixed points are landing points of periodic dynamic rays. Using these results and the fact that dynamic rays move holomorphically wherever they are well defined, we obtain the following.

\begin{prop}[Misiurewicz wakes]\label{Misiurewicz wakes}
Let $c$ be a Misiurewicz parameter belonging to a parabolic wake $\WW$ attached to $W_0$. Let  $G_\suno\ldots G_{\s_q}$ be the parameter rays landing at $c$, $G_{\s^+}$ and $G_{\s^-}$ be the characteristic parameter rays bounding $\WW$.
The connected components of $\WW\setminus\bigcup_i G_{\s_i}$ are called \emph{Misiurewicz wakes}. The dynamic rays $g^c_{\s_1}\ldots g^c_{\s_q}$  move holomorphically in each Misiurewicz wake, and they land together in the dynamical plane for $c$ if and only if $c$ belongs to  one of the Misiurewicz wakes  not  containing $G_{\sp}$ and $G_{\sm}$ in its boundary.
\end{prop}

\subsection{Fibers}


The \emph{extended (parameter) fiber}  of a parameter $c_0$ with all periodic orbits repelling is the set of parameters which cannot be separated from $c_0$ by a pair of parameter rays with periodic or preperiodic  addresses landing together. Observe that by continuity of the map which associates to a parameter to the multiplier of a given orbit, all parameters with an indifferent cycle belong to the boundaries of hyperbolic components.  The \emph{reduced (parameter) fiber} of a parameter $c_0$ is the extended fiber intersected with the set of  non-escaping parameters. For parameters with indifferent or attracting periodic orbits, also  curves formed only by parameter rays, hyperbolic parameters and finitely many parabolic parameters have to be considered as separation lines (see \cite{RS3} for alternative definitions). 

\noindent The \emph{dynamical fiber} of a point $z_0$ is the set of points which cannot be separated from $z_0$ by a pair of dynamic  rays with  periodic  or preperiodic addresses landing together. 

\noindent A fiber is \emph{trivial} if it contains at most one non-escaping parameter or point.

\noindent Fibers cannot contain the accumulation sets of arbitrarily many parameter rays (see [\cite{RS3}, Lemma 17]):
\begin{lem}\label{17}
Any extended fiber contains the accumulation set of at most finitely many parameter rays, and these parameter rays have bounded addresses.
\end{lem}

\section{Puzzles and Parapuzzles}\label{Puzzles}
 Puzzles are collections of progressively finer  Markov partitions of the dynamical plane, which can be used to study the dynamics on the Julia set via symbolic dynamics. Similarly, parapuzzles are progressively finer partitions of the parameter space which are associated to the puzzles constructed in dynamical plane, hence giving a combinatorial description of the parameters in them.

 The  construction of puzzles  for polynomials (usually called Yoccoz Puzzle), as well as the construction of parapuzzles,  has turned out to be very useful in proving rigidity results.  Puzzles for polynomials  have been introduced in \cite{BH}, \cite{Hu} and \cite{Mi2}. Various versions  have been constructed also for other maps, for example rational maps (\cite{Ro}). 
In this section we introduce for the first time the construction of puzzles and parapuzzles for the exponential family, hoping that this construction might turn useful for further advances in the field.

\subsection{Puzzles and combinatorial non-recurrence}\label{Puzzles 1}

{From now on we assume $c$ to be a nonescaping parameter and that there are no preperiodic rays landing at $c$. See Remark~\ref{Escaping puzzle} for some comments about this assumptions.} {Since the case in which $c$ is the landing point of a preperiodic rays has been treated in \cite{Be}, and we stated the Rigidity Conjecture in terms of reduced fibers, these assumptions are no loss of generality.} 

Let $\Gamma$ be a closed forward invariant  graph formed by finitely many periodic rays together with their landing points, and consider the connected components of $\C\setminus \Gamma$. 
For each $n$, the sets  $$\Gamma_n:=\bigcup_{j=0}^{n}f^{-j}(\Gamma)$$ also partition $\C$ into countably many components  (see Figure~\ref{Puzzle Figure} for an example of puzzle).
 
\begin{defn}
 The countable collection of connected components of $\C\setminus\Gamma_n$ is called the \emph{puzzle of level $n$} (induced by $\Gamma$); the connected components themselves are called \emph{puzzle pieces of level $n$} and are denoted by $Y^{(n)}_j$.
 For each $n\in\N$ there is exactly one puzzle piece of level $n$ containing a left half plane, which is called the \emph{branching puzzle piece} and denoted by $Y^n_*$. There is also exactly one \emph{singular puzzle piece} containing $c$, and the branching puzzle piece of level $n+1$ is the preimage of the singular puzzle piece of level $n$. 
 \end{defn}

If $c$ is non-escaping {and is not the landing point of a preperiodic ray which maps to $\Gamma$ under finitely many iterations}, it never belongs to $\Gamma_n$, hence at each level  $c$ belongs to the interior of some puzzle piece.

 \begin{figure}[hbt!]
\begin{center}
\def\svgwidth{13cm}

\begingroup
  \makeatletter
  \providecommand\color[2][]{%
    
    \renewcommand\color[2][]{}%
  }
  \providecommand\transparent[1]{%
   
    \renewcommand\transparent[1]{}%
  }
  \providecommand\rotatebox[2]{#2}
  \ifx\svgwidth\undefined
    \setlength{\unitlength}{1998.703568pt}
  \else
    \setlength{\unitlength}{\svgwidth}
  \fi
  \global\let\svgwidth\undefined
  \makeatother
  \begin{picture}(1,0.48994462)%
    \put(0,0){\includegraphics[width=\unitlength]{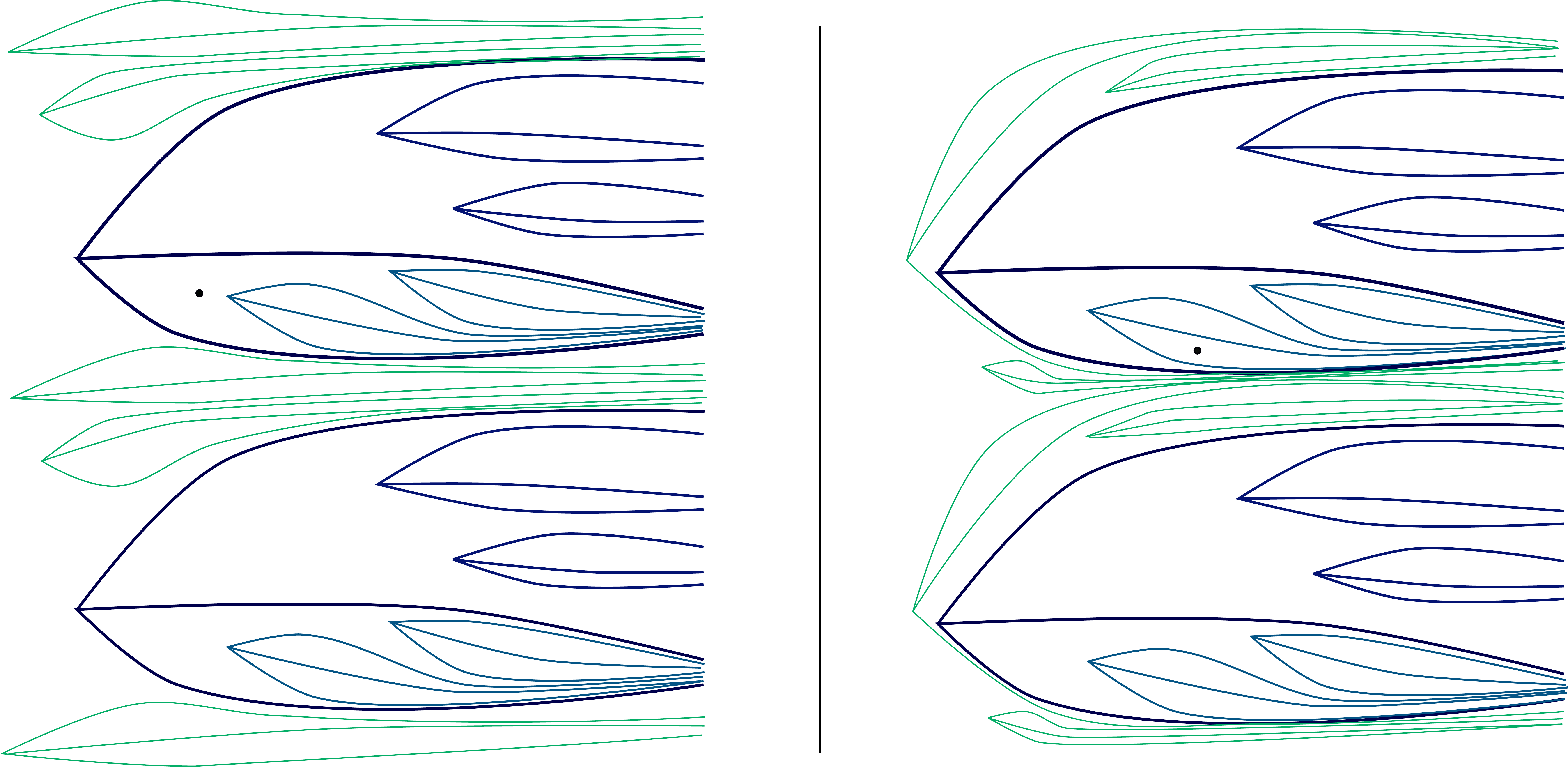}}%
    \put(0.13785413,0.31056612){\color[rgb]{0,0,0}\makebox(0,0)[lb]{\smash{$c$}}}%
    \put(0.74738295,0.27149743){\color[rgb]{0,0,0}\makebox(0,0)[lb]{\smash{$c$}}}%
  \end{picture}%
\endgroup
\end{center}
\caption{\small Two possible topological configurations for the puzzle of level 4 with $\Gamma$ formed by 3 rays of period 3 landing at a fixed point. The actual configuration of puzzle pieces of a certain level  depends on the position of $c$ with respect to the puzzle pieces of lower level. The configuration of puzzle pieces is invariant under translation by $2\pi i $. As the level increases, the new puzzle pieces are drawn in a paler color.}
\label{Puzzle Figure}
\end{figure}

The singular value $c$ is called  \emph{combinatorially non-recurrent} if there exists some forward invariant graph $\Gamma$ (formed by the closure of finitely many periodic dynamic rays) and some $n$ such that the  connected component of $\C\setminus\Gamma_n$ containing  $c$ does not intersect $\PP(f)$. 

Observe that parabolic and hyperbolic parameters are combinatorially recurrent. It is also known that the boundary of Siegel Disks is contained in the $\omega$-limit set of a recurrent singular value (see \cite{RvS}, Corollary 2.10).  So, in our setting,  the condition of combinatorial non-recurrence implies that $J(f)=\C$. Combinatorial non-recurrence is a more restrictive condition than non-recurrence; however, we show in Section~\ref{Non-recurrence and combinatorial non-recurrence} that  when $\PP(f)$ is bounded and contained in the Julia set, non-recurrence implies combinatorial non-recurrence.

\begin{rem}\label{Escaping puzzle}.
\begin{enumerate}
\item If $c$ is a non-escaping parameter belonging to a parabolic wake attached to the period one hyperbolic component $W_0$, there is a repelling fixed point  $\alpha=\alpha(c)$, defined as an analytic continuation of the attracting fixed  point, which is the landing point of  finitely many periodic rays permuted transitively by the dynamics. These are the rays used for the most traditional puzzle construction. We use a more general $\Gamma$ in order to allow for a weaker form of combinatorial non-recurrence (so that Theorem D holds). It is not known yet whether all non-escaping parameters are contained in such a wake, however, this is expected to be the case.  A sufficient condition for this to happen is that $\PP(f)$ is bounded (see \cite{BL}, Corollary 4.14).

 \item It is possible to define the label $j$ for puzzle pieces of level $n$ by using sequences in $\Z^n$, in a way that is consistent with the dynamics and with the vertical order of the dynamic rays at infinity. However, this takes some extra work and it is not needed for the rest of the paper; for simplicity, we will consider $j\in\N\cup\{*\}$. 
 
\item If $c$ is an escaping parameter or  if $c$ is the landing point of a  preperiodic dynamic ray, a puzzle constructed from a graph $\Gamma$ is \emph{allowable} if $\{f^n(c)\}_{n\in\N}\cap\Gamma=\emptyset$. Such a parameter is combinatorially nonrecurrent if there exists an allowable  puzzle satisfying the notion of combinatorial nonrecurrence. {Since the case in which $c$ is the landing point of a preperiodic rays has been treated in \cite{Be}, and we stated the Rigidity Conjecture in terms of reduced fibers, these assumptions are no loss of generality.}

\end{enumerate}
\end{rem}

The next lemma follows directly from the definition of puzzle pieces, and the basic properties of dynamic rays.

\begin{lem}\label{Puzzle dynamics} The collections $\{Y^{(n)}_j\}$ define a puzzle for $f$, i.e. 

\begin{itemize}
 \item  
 $\forall\; n\in\N, i,j\in\N\cup\{*\}$, either $\Ynu_j \subset Y^{(n)}_i$ or $\Ynu_j \cap Y^{(n)}_i=\emptyset$.
\item $\forall\;  n,i\in\N$, $Y^{(n)}_i$ maps univalently to  $Y^{(n-1)}_j$ for some $j$ depending on $i$;
\item $Y^{(n)}_*$ is mapped to the singular puzzle piece of level $n-1$ as an infinite degree covering.
\end{itemize}
\end{lem}

\begin{proof}
Boundaries of puzzle pieces are formed by dynamic rays together with their landing points, hence non-nested puzzle pieces do not intersect.
Also, the boundaries of puzzle pieces of level $n+1$ contain all  preimages of the boundaries of puzzle pieces of level $n$, so the image of a puzzle piece cannot intersect two different puzzle pieces.
\end{proof}

\subsection{Parapuzzles and combinatorial equivalence}\label{Parapuzzles}  

 Let $c_0$ be a {non-escaping} parameter {which is not the landing point of a preperiodic ray.} Let  $\Gamma$ be a forward invariant graph formed by finitely many periodic rays together with their landing points as in Section~\ref{Puzzles 1},  and  $\AA_0 =\{\s_i\}_{i=1}^q$ be the set of addresses of the  dynamic rays in $\Gamma$. Define

 $$\AA_N:=\{\s\in\SS:\sigma^j(\s)\in\AA_0 \text{ for some } j\leq N\}.$$ 
The boundaries of puzzle pieces of level $N$ consist exactly of dynamic rays with addresses in $\AA_N$; these rays are either periodic or preperiodic with preperiod $k\leq N$. Hence, by Propositions~\ref{Parabolic wakes} and \ref{Misiurewicz wakes}, the boundaries of puzzle pieces of level $N$ move holomorphically in $\C\setminus \bigcup_{\s\in\AA_N} \ov{G_\s}$.

A \emph{parapuzzle piece} of level $N$ for $\Gamma$ is a region in parameter space over which the boundaries of puzzle pieces of level $N$ move holomorphically.
 
Observe that the boundary of  a parapuzzle piece of level $N$ consists of countably many parameter rays with addresses in $\AA_N$ together with their landing points, hence of either parameter rays landing at Misiurewicz parameters or  characteristic  parameter rays landing at parabolic parameters (see \cite{Sc0} for landing of parameter rays with periodic and preperiodic addresses). Also, parapuzzle pieces of different levels are either disjoint or contained one inside the other.

Two exponential  maps $f_c,f_\cprime$ are \emph{combinatorially equivalent} up to level $N$ (with respect to the parapuzzle induced by $\Gamma$) if they belong to the same parapuzzle piece up to level $N$. If two maps are combinatorially equivalent as in Section~\ref{Dynamic rays}, by Theorems~\ref{Parabolic wakes} and \ref{Misiurewicz wakes} they  belong to the same parapuzzle piece for all levels independently of the choice of $\Gamma$, {as long as $\Gamma$ is allowable for both parameters in question. The condition of allowability is automatically satisfied if both parameters are neither escaping nor preperiodic}.  For simplicity in this paper, we will only be concerned with non-escaping parameters.

Two puzzle pieces $Y^{(n)}_\ell$ for $f_c$ and  ${Y^{(n)}_{\ell}}'$ for $f_\cprime$ are  \emph{equivalent} if they are bounded by  dynamic rays with the same addresses which land together in the same pattern (see again Definition~\ref{Combinatorial equivalence def} ). If two maps are combinatorially equivalent up to level $N$, for each $n<N$ and for each puzzle piece  $Y^{(n)}_\ell$ for $f_c$ there is exactly one equivalent puzzle piece  ${Y_\ell^{(n)}}'$ for $f_\cprime$. Equivalent puzzle pieces will be labelled with the same label.

\begin{prop}[Combinatorial equivalence]\label{Combinatorial equivalence}
If two exponential maps $f_c,f_\cprime$ are combinatorially equivalent up to level $N$, their singular values  belong to equivalent puzzle pieces up to level $N-1$. 
\end{prop}

\begin{proof}
Suppose by contradiction that there is a level $n\leq N-1$ such that $c,\cprime$ belong to non-equivalent puzzle pieces. Then there are two addresses $\s,\s'$ such that the curve $\ov{g^c_\s\cup g^c_{\s'}}$ encloses $c$, but the curve $\ov{g^\cprime_\s\cup g^\cprime_{\s'}}$ does not enclose $\cprime$.
By Proposition~\ref{Splitting}, the puzzle pieces of level $n+1$ which contain preimages of $g^c_\s, g^c_{\s'}$ on their boundary cannot be equivalent to the puzzle pieces of level $n+1$ which contain preimages of $g^\cprime_\s, g^\cprime_{\s'}$; this  contradicts combinatorial equivalence at level $n+1$.
\end{proof}

{\begin{cor}
If two exponential maps $f_c,f_\cprime$ are combinatorially equivalent then for any $n,j\in\N$, $\fc^{j}(c)\in Y^{(n)}_\ell$ if and only if  $\fcp^{j}(c')\in {Y^{(n)}}'_\ell$.
\end{cor}

\begin{proof}
Fix $n,j\in\N$. By Proposition~\ref{Combinatorial equivalence} $c$ and $c'$ belong to equivalent puzzle pieces of level $n+j$, and since  by definition  two such puzzle pieces are mapped by $f^j$ to two puzzle pieces which are also equivalent, the claim follows.
\end{proof}

}

\section{Quasiconformal Rigidity: proof of Theorem B}\label{qc rigidity}

In  this section we prove Theorem B. 
We first construct a quasiconformal map $\psi_N$ between $f_c$ and $f_\cprime$, which is a conjugacy on the boundary of puzzle pieces up to a finite level $N$; then $\psi_N$  is used as the initial map in a  lifting procedure to  obtain  a quasiconformal map $\Psi$ which is a conjugacy on $\PP(f)$; finally,  $\Psi$ is used in a new lifting procedure to obtain a quasiconformal map   $\Phi$  which is a conjugacy on all  preimages of $\PP(f)$. In fact, because the latter are dense, by continuity $\Phi$ is a  quasiconformal conjugacy on the entire plane.

{Recall that $\AA_k$ is the set of addresses of the rays which form the boundaries of puzzle pieces of level $k$.}

\begin{prop}[Initial quasiconformal map]\label{Initial quasiconformal map}
Let $c,\cprime$ be non-escaping and let $f_c$, $f_\cprime$ be combinatorially equivalent up to  level $N$ for some $N\in\N$.  Then there exists a quasiconformal  map $\psi_N: \C\ra\C$ such that
\begin{itemize}
\item $\psi_N(c)=\cprime$ and $\psi_N(\fc(c))=\fcp(\cprime)$ ;
\item $\psi_N$ is a conjugacy on the dynamic rays $g^c_\s$ with {$\s\in\AA_N$}.
\end{itemize} 
\end{prop}

\begin{proof}
As $\fc, \fcp$ are combinatorially equivalent up to level $N$, there is a  parapuzzle piece $\Lambda$ of level $N$ which contains both $c$ and $\cprime$.  
By Propositions~\ref{Parabolic wakes} and \ref{Misiurewicz wakes} together with Theorem~\ref{ExistenceRays}, the map $h_{\ctilde}:=g_\s^\ctilde\circ(g_\s^c)^{-1}$ defines a holomorphic motion  over $\Lambda$ of the dynamic rays $g^c_\s$ for  $\s\in\AA_k$ with $k\leq N$.  By definition,  $h_{\ctilde}$ is a conjugacy on the dynamic rays $g^c_\s$ with $\s\in\AA_k$ for all $k\leq N$. 
The singular value $c$ itself also moves holomorphically in $\Lambda$, as well as its first image $\fc(c)$, so $h_{\ctilde}$  can be defined   $h_{\ctilde}(c)=\cprime$ and $h_{\ctilde}(\fc(c))=\fcp(\cprime)$.
By the Bers-Royden extension (Theorem~\ref{Lambda Lemma}), $h_{\ctilde}$ can be extended to a quasiconformal  map $\psi^{\ctilde}:\ \C\rightarrow\C$ for any $\ctilde\in\Lambda$. In particular, $\psi_N:=\psi^{c'}$  is a conjugacy on the boundaries of the infinitely many puzzle pieces of level $N$ and $\psi_N(c)=c'$.
\end{proof}

From now on, let $\Gamma$ be the forward invariant graph for which $c$ is combinatorially non-recurrent, and consider the puzzle induced by $\Gamma$.

\begin{prop}[Conjugacy on $\PPfc$]\label{Conjugacy on P}
Let $c,\cprime$ be non-escaping and  let $\fc,\fcp$ be two combinatorially equivalent maps, such that $c$ is combinatorially non-recurrent. Then there exists a quasiconformal  map $\psi:\C\ra\C$ which is a conjugacy on $\PPfc$. {Moreover, $\psi$ is  a  conjugacy between $\fc$ and $\fcp$ on the boundaries of the infinitely many puzzle pieces of level $N$, and maps $c$ to $\cprime$.}
\end{prop}

\begin{proof}
For each level $n$, denote by $\Pn$  the union of the puzzle pieces of level $n$ intersecting $\PP(f_c)$. 
Observe that $\Pnp\subset \Pn$, and that by forward invariance of $\PPfc$ if a puzzle piece $Y\subset\Pnp$, then $f_c(Y)\subset{\Pn}$.
As $f_c$ is combinatorially non-recurrent, there is some level $N$ such that $c\notin P^{(N)}$.
Let $\psi_N$ be as in Proposition~\ref{Initial quasiconformal map} and  define inductively the sequence $\{\psi_n\}_{n\geq N}$ as:

\begin{align}\label{defpsi}
\psi_{n+1}&=\psi_n  &\text{ on }\C\setminus P^{(n)};  \\
\psi_{n+1}&=f^{-1}_\cprime\circ \psi_n \circ f_c &\text{ on $\ov{\Yl}$, for any $\Yl=\Yl^{(n+1)}\subset P^{(n)}$.}\label{psi}
\end{align}

Observe that we are redefining $\psi_n$ on all puzzle pieces of level $n+1$ which are contained in $\Pn$, not only on the ones which are also contained in $\Pnp$. This is needed in order to ensure continuity. The branch of the inverse $f^{-1}_\cprime$ is defined so that  $ f^{-1}_\cprime\circ \psi_n \circ f_c (\ov{\Yl})= \ov{\Yl}'$, where ${\Yl}'$ is  the puzzle piece for $f_\cprime$ which is equivalent to $\Yl$. We show  by induction that this is possible while we prove the properties of  $\{\psi_n\}$.

Now we show that the  maps $\{\psin\}_{n\geq N}$ are well defined and satisfy the following properties:
\begin{enumerate}
\item[1.]  $\psinp=\gscp\circ(\gsc)^{-1}$ on any $\gsc\in\partial Y^{(j+1)}$ such that $Y^{(j+1)}\subset P^{(j)}$ for some $j\leq n$.
\item[2.] $\psinp$ is $K$-quasiconformal.       
\item[3.] $\fcp\circ\psinp=\psin\circ\fc$ on $\Pnp$.
\end{enumerate}
{Since $\psinp$ is only modified outside $\Pn$,  for each $n$ the map $\psinp$ remains a  conjugacy on the boundaries of the infinitely many puzzle pieces of level $N$ and moreover  $\psinp(c)=c'$, hence these properties are inherited by any limiting map.}  
Observe that $\psi_N$ satisfies all of the properties above for some quasiconformality constant $K$.  We show by  induction that they are satisfied for all $n>N$.

We first show that $\psinp$ is well defined on every $\Yl^{(n+1)}\subset P^{(n)}$. For any puzzle piece $\Yl^{(n+1)}$, its  image is $\fc(\ov{\Yl^{(n+1)}})=\ov{\Ylt^{(n)}}$ for some puzzle piece $\Ylt^{(n)}$ of level $n$.
 Let $\Yltp^{(n)}$ be the  puzzle piece which is equivalent to $\Ylt^{(n)}$.
The puzzle piece $\ov{\Ylt^{(n)}}$   does not contain $c$ by combinatorial non-recurrence, so $\ov{\Yltp^{(n)}}$ does not  contain $c'$ by Proposition~\ref{Combinatorial equivalence}. In particular inverse branches for $\fcp$ are well defined and univalent in a neighborhood of $\ov{\Yltp^{(n)}}$, and there is a univalent branch $\fcp^{-1}:\Yltp^{(n)}\ra\Ylp^{(n+1)}$. So  $\psi_{n+1}$ as in (\ref{psi}) is well defined using the univalent  branch of the inverse described here for any puzzle piece under consideration.

We now show Property $1$. Let $g^c_\st$ be a dynamic ray such that $g^c_\st\in\partial Y^{(j+1)}$ for some  $Y^{(j+1)}\subset P^{(j)}$ and some $j\leq n$. Then $g^c_{\sigma\st}\in\partial Y^{(j)}$ for some  $Y^{(j)}\subset P^{(j-1)}$. By direct computation and by the induction hypothesis, on $g^c_\st$, 
\begin{align}\label{Calculation}
\psinp|_{g^c_\st} &= \fcp^{-1}\circ\psin\circ\fc|_{g^c_\st}=\\
                &= \fcp^{-1}\circ g^\cprime_{\sigma\st}\circ (g^c_{\sigma\st})^{-1}\circ\fc|_{g^c_\st}\\
                &=g^{\cprime}_\st\circ (g^{c}_\st)^{-1}|_{g^c_\st}.
\end{align}
 
(The second equality follows from Property 1. in the induction hypothesis, while the last equality comes from the functional equation in Theorem~\ref{ExistenceRays} and the choice of branch of $\fcp^{-1}$). 
In particular, by the induction hypothesis $\psinp=\psin$ on any   $g^c_\st\in\partial Y^{(j+1)}$ if  $Y^{(j+1)}\subset P^{(j)}$ for $j\leq n-1$.
 
Property 1. together with the definition of $\psinp$ implies  continuity: on rays belonging to $\partial P^{(n)}$ we obtain that $\psinp=\psin$, while on any ray $g^c_\s$ in the boundary of two puzzle pieces of level $n+1$ the new defined functions match because $\psinp$ coincides with   $\gscp\circ(\gsc)^{-1}$.  So $\psinp$ is a homeomorphism; to   show that it  is $K$-quasiconformal, note that whenever $\psinp$ is redefined, it is done by pre- and post-composing with a  conformal map, so the dilatation remains $K$ and $\psinp$ is  $K$-quasiconformal (as it is a homeomorphism which is  $K$-quasiconformal outside a countable number of smooth curves). This proves   Property $2.$.  Property $3.$ follows directly from the definition of $\psinp$ because $P^{(n+1)}\subset P^{(n)}$.
 
 
By Property $2.$ the functions $\psi_n$  are uniformly quasiconformal, can be extended as to fix infinity and by Property $3.$ they all coincide on $\partial P^{(N)}$, so by Lemma~\ref{Precompactness} there exists a  $K$-quasiconformal limit function $\psi: \hat{\C}\ra\hat{\C}$ fixing infinity. {All limit functions coincide outside the set $Q:=\cap_n \ov{P^{(n)}}$, which is forward invariant, closed and omits a neighborhood of $c$ hence  has zero area by Corollary~\ref{zero area}. So all limit functions coincide outside a set of measure zero hence  the limit is unique.} By Property $3.$, the limit map $\psi$ is a conjugacy between $f_c$ and $f_\cprime$ on $\underset{n\geq N}{\bigcap} P^{(n)} \supset\PP(f)$.
{Since for each $n$ $\psi_n$ maps a puzzle piece in $P^{(n)}$ to its equivalent puzzle piece for $\fcp$, by Proposition~\ref{Combinatorial equivalence} the limiting conjugacy maps $\PP(\fc)$ to $\PP(\fcp)$.}
\end{proof}

We conclude the proof of Theorem B by lifting $\psi$ to a continuous map which is a  conjugacy on the preimages of the postsingular set, and then  use the fact that the latter are dense to obtain by continuity a conjugacy on the entire plane.
The next theorem is one of the fundamental facts in the theory of covering spaces (see e.g. \cite{Ha}, Proposition 1.33).

\begin{thm}[Lifting]\label{Lifting}
Let $X,Y,Z$ be topological spaces with base points $x,y,z$ respectively. If $X$ is a covering of $Z$ via a map $h$ such that $h(x)=z$, and $f: Y\ra Z$ is a continuous map with $f(y)=z$ and $Y$ is simply connected, then there exists a unique continuous lift $\tilde{f}: Y\ra X$ with $\tilde{f}(y)=x$.
\end{thm}

\begin{proof}[Proof of Theorem B] 

The parameters  $c, c'$ are non-escaping and, since they are combinatorially equivalent, they cannot be Misurewicz parameters since otherwise they would coincide by the main result in \cite{Be}. Hence the set $\PP(\fc)$ is not discrete.  
Let $\Psi_0$ be $\psi$ as given by Proposition \ref{Conjugacy on P}. Consider the lifting diagram below (where the couples $(X,x)$ mean the space $X$ with base point $x$):  
\begin{displaymath}
\begin{array}{ccc}
   & \Psi_{n+1}  &   \\
  (\C,c) &  \longrightarrow & (\C,\cprime)   \\
   f_c\downarrow &   & \downarrow f_{\cprime} \\
 (\C\setminus\{c\},\; f_c(c)) & \longrightarrow & (\C\setminus\{c'\},\; f_\cprime (\cprime))  \\
    & \Psi_n &   \\
\end{array}
\end{displaymath}

As $\C$ is simply connected, the existence of the map $\Psi_{n+1}$ as lift of $ \fc\circ\Psi_n$ is ensured by Theorem~\ref{Lifting}, and by choice of base points, for each $n$, $\Psinp(c)=\cprime$. Recall that $\Gamma_n$ is the graph forming the boundary of puzzle pieces of level $n$. 
We show by induction that the $\Psi_n$ satisfy the following properties:
\begin{itemize}
\item[1.]  $\Psinp=\Psi_n$ on $\fc^{-j}(\PPfc)\cup\Gamma_{N+n}$  for $j\leq n$.
\item[2.] $\Psinp$ is a conjugacy on $\bigcup_{j\leq {n+1}}\fc^{-j}(\PPfc)\cup\Gamma_{N+n+1}$.
\end{itemize}

 All properties are verified for $n=0$ by Proposition~\ref{Conjugacy on P}. 

We now show Property 1.  Since  $\Psi_n$ is a conjugacy on $\bigcup_{j\leq n}f_c^{-n}(\PP(f_c))$ we have that 
\begin{equation}\label{Coincide}
\Psi_{n+1}(z)=(f_\cprime^{-1}\circ\Psi_{n}\circ f_c)(z)=(f_\cprime^{-1}\circ f_\cprime\circ\Psi_{n})(z)
\end{equation} 

where $\fcp^{-1}$ is the branch of the inverse  induced locally  by the definition on $\Psinp$. Since the  preimages of a point   under the exponential map are all vertical translates of each other by integer multiples of $2\pi$ we have that 

\begin{equation}
\Psi_{n+1}(z)=(f_\cprime ^{-1}\circ f_\cprime\circ\Psi_{n})(z)= \Psin (z)+ 2\pi i k(z).  \end{equation}

{Using the fact that  $\Psinp$ is a lift, that $\Psinp(c)=\cprime$, and  that $\Psinp$ is a conjugacy on the boundary of puzzle pieces of level $N+n$ (which contain countably many curves spaced by $2\pi i $), we have that    $k=k(z)$ does not depend on $z$ and in fact $k=0$. This  gives Property 1.}

 In particular, $\Psi_{n+1}$ is still a conjugacy on $\bigcup_{j\leq n}{f_c^{-j}\PP(f_c)}$. 
To show that $\Psi_{n+1}$ is a conjugacy also on ${f_c^{-n-1}\PP(f_c)}$ (hence proving Property 2.), let $z\in{f_c^{-n-1}\PP(f_c)}$; by definition 

\[ f_\cprime\circ \Psi_{n+1}=\Psi_n\circ f_c=\Psi_{n+1}\circ f_c \]
by Property 2, because $f_c(z)\in {f_c^{-n}\PP(f_c)}$. 

As locally  $\Psi_{n+1}$ is obtained by pre- and post-composing $\Psi_n$ with conformal maps, all $\Psi_n$ are uniformly quasiconformal, and by Equation~\ref{Coincide}  they all coincide on the postsingular set.
Extending the $\Psi_n$ to fix  infinity,   Lemma~\ref{Precompactness} gives a limit map $\Psi$ which is a conjugacy on $\bigcup_{n\in\N} f_c^{-n}(\PP(f)).$  The latter ones are dense, so by continuity $\Psi$ is a conjugacy on all of $\C$.

\end{proof}


\section{Conformal Rigidity: proof of Theorem C}
\label{Rigidity}

When the postsingular set is bounded, the proof of Theorem C can be made relatively easy by using Theorem~\ref{Absence of line fields}. 

\begin{proof}[Proof of Theorem C for bounded postsingular set]
Let $\Psi$ be the  quasiconformal conjugacy obtained in Theorem B, and  $\sigma_0$ be the standard conformal structure on $\C$. The pushforward of $\sigma_0$ by $\Psi$ defines  an invariant conformal structure $\sigma'$ in the dynamical plane for $f_\cprime$. The conformal structure $\sigma'$ defines an invariant line field, which is constant by Theorem~\ref{Absence of line fields}. So $\sigma'$ is the standard conformal structure and $\Psi$ is conformal by Weyl's Lemma, hence $c'=c+2\pi i n$ for some $n$. As $c,c'$ are combinatorially equivalent, they are in the same fiber, and $c=c'$.
\end{proof}

 We now use an  open-closed argument to show that the reduced fibers of combinatorially non-recurrent parameters are trivial also when the postsingular set is not bounded. 
Let $\QC(c)$ be the quasiconformal class of $c$, that is the  connected component containing $c$ of the set of parameters $c'$ such that $f_\cprime$ is quasiconformally conjugate to $f_c$. Observe that parameters of the form $c+2\pi i n$ cannot belong to $\QC(c)$ unless $n=0$.

\begin{lem}[Quasiconformal classes are open]\label{Quasiconformal classes are open}
 Let $c$ be non-escaping and combinatorially non-recurrent. If $\QC(c)\neq\{c\}$, $\QC(c)$ is open.
\end{lem}

\begin{proof}
Let $c\neq \cprime\in \QC(c)$, $\Psi$ be a quasiconformal conjugacy between $f_c$ and $f_\cprime$. Let $\mu_0=0$ be the Beltrami coefficient for the  standard conformal structure, $\mu':=\Psi_*\mu$ be the Beltrami coefficient obtained in the dynamical plane for  $f_\cprime$, and for $\lambda\in\D$ let  $\mu_\lambda$ over $\lambda\in \D$ be an analytic interpolation between $\mu$ and $\mu'$ (for example, $\mu_\lambda=\lambda\mu'$). If $\mu'\neq0$, hence the deformation is non-trivial, by the Measurable Riemann Mapping Theorem to each  $\mu_\lambda$ corresponds a quasiconformal map $\Psi_\lambda$  conjugating $f_c$ to an exponential map $f_\lambda$. As $\mu_\lambda$ depends holomorphically on $\lambda$, the $\Psi_\lambda$, and hence $f_\lambda$, depend holomorphically on $\lambda$; in particular there is an open neighborhood of $c$ on which all maps are quasiconformally conjugate to $f_c$. With the same argument, it is possible to find an open neighborhood contained in $\QC(c)$ for any $\ctilde\in\QC(c)$.
\end{proof}

A maximal open set of parameters all of which are topologically conjugate to each other, and which are not hyperbolic, is called a \emph{non-hyperbolic component}.
Any such component $\QQ$ is simply connected; otherwise, there would be bounded components of $\C\setminus\ov{\QQ}$, contradicting the facts that escaping parameters are dense in the bifurcation locus and that hyperbolic components are unbounded. 
\begin{lem}[Boundaries of non-hyperbolic components]\label{Boundaries of non-hyperbolic components}
Consider the exponential family.
The boundary of a non-hyperbolic component $\QQ$ in parameter space cannot contain escaping parameters which are accessible from the inside of $\QQ$. In particular, $\QQ$ contains infinitely many non-escaping parameters on its boundary.
\end{lem}

\begin{proof}
 As $\QQ$ is simply connected it can be uniformized to the unit disk, and by the boundary behavior of the Riemann map, there is only a set of measure  zero  of parameters on $\partial \QQ$ which are not accessible from the inside of $\QQ$.   Also, non-hyperbolic components are fully contained in a single fiber because they cannot intersect parameter rays, hence by Lemma~\ref{17}, $\ov{\QQ}$ intersects only finitely many parameter rays. So there are infinitely many parameters on $\partial\QQ$ which are accessible via  a curve $\gamma:[0,1)\ra\QQ$, and which are not endpoints of parameter rays. To prove the claim it is enough to show that no such accessible parameter can be escaping. 
Suppose by contradiction that $c_0\in\partial \QQ$ is an escaping parameter accessible from the inside of $\QQ$, and is not the endpoint of a parameter ray. Then $c_0= G_\szero(t_0)$ for some parameter ray $G_\szero$, and there exists an arc  $G_\szero(t_0-\e,t_0+\e)$ which can be oriented following increasing $t$. 
 Also, $G_\szero$ is approximated on compact sets on  both sides by curves in hyperbolic components (see Section~\ref{Parameter rays, parabolic wakes, Misiurewicz wakes}); as $\gamma$ is a local transversal to the arc $G_\szero(t_0-\e,t_0+\e)$ at $c_0$, infinitely many of these curves have to intersect $\gamma$, which is impossible because $\QQ$ is a non-hyperbolic component. 
\end{proof}

\begin{rem}One could object that there could be escaping points on $\partial\QQ$ which are endpoints of parameter rays rather than points on a parameter ray. However any parameter ray which lands on an escaping parameter has an unbounded address (see \cite{FRS}), hence cannot accumulate on $\QQ$ by Lemma \ref{17}. 
\end{rem}

\begin{proof}[Proof of Theorem C]Suppose by contradiction that there are two parameters $c\neq c'$  non-escaping and combinatorially equivalent, with  $c$ combinatorially non-recurrent. Then $c'\in\QC(c)$ by Theorem B, and $c\in\ov{\QQ(c)}$ for some non-hyperbolic component $\QQ(c)$, because quasi-conformal conjugacy implies topological conjugacy, and because combinatorial non-recurrence implies that $c$ is not hyperbolic. For the same reason, $\QC(c)\subset\ov{\QQ(c)}$, and since $\QC(c)$ is open,  $\QC(c)\subset{\QQ(c)}$  . Let  $F_R(c)$ denote the reduced fiber of $c$.
As parameter rays cannot intersect non-hyperbolic components, $\QQ(c)\subset F_R(c)$. On the other side, by Theorem B, $F_R(c)\subset\QC(c)$, hence $\QC(c)=\QQ(c)= F_R(c)$. $F_R(c)$ is open by Lemma \ref{Quasiconformal classes are open}, while it contains at least one of its boundary points by Lemma \ref{Boundaries of non-hyperbolic components}, giving a contradiction.
\end{proof}

If two parameters are in the same fiber, then they are combinatorially equivalent, so Theorem C implies that reduced fibers of combinatorially non-recurrent parameters are trivial.

\section{Non-recurrence and combinatorial non-recurrence}\label{Non-recurrence and combinatorial non-recurrence}

In this section we show that for parameters with bounded postsingular set, non-recurrence implies combinatorial non-recurrence provided they are not hyperbolic or parabolic parameters. We achieve this through  the combinatorial similarity between the parameter plane for exponentials and for unicritical polynomials. We use several results about polynomial dynamics which are most likely known to the reader, like the analogues of the theorems described for exponentials in Section~\ref{Puzzles} (see e.g. \cite{Mi}). We add references for the less well known results.
To introduce notation, we recall the  definitions of dynamic and parameter rays for polynomials.

Let $P^D_c(z)=z^D+c$ be a unicritical polynomial of degree $D$. If $c$ is non-escaping, the complement of the filled Julia set $K_c$ can be uniformized by $\C\setminus\ov{\D}$ via the B\"ottcher map, which is also a  conjugacy between $P^D_c$ and $P^D_0$ in a neighborhood of infinity. The angles of the straight rays in $\C\setminus\ov{\D}$ can be expressed as sequences $\s$ over an alphabet of $D$ symbols, containing integer entries $s_i\in (-D/2 +1/2,+D/2+1/2)$  if $D$ is even, and $s_i\in (-D/2,+D/2)$  if $D$ is odd.
The preimage of a straight ray of angle $\s$ under the B\"ottcher map is called the \emph{dynamic ray} of angle $\s$ and denoted by $g_\s^{D,c}$ or just $g_\s^{c}$ when the degree is implicit.
The dynamics of $P^D_c$ on the dynamic rays is conjugate to the dynamics of the shift map $\sigma=\sigma_D$ over the sequences over $D$ symbols. 
Similarly, the complement of the connectedness locus $M_D$ for the family $\{P^D_c\}$ can be uniformized by $\C\setminus\ov{\D}$, and the preimage of a straight ray of angle $\s$ is called the \emph{ parameter ray}  $G^D_\s$ of angle $\s$. Near infinity, both dynamic and parameter rays respect the \emph{cyclic order} induced by the cyclic order on the set of sequences over $D$ symbols  identified with the unit circle.

\subsection{Proof of Theorem D}
For the proof of Theorem D we need some additional results about the exponential family as well as some rather specific knowledge about the parameter structure of unicritical polynomials. We tried to make this section as self-contained as possible, avoiding however to dwell excessively in  the theory of renormalization for unicritical  polynomials.

We make use of the following two results from \cite{BL} ( see \cite[Corollary 4.6]{BL}  and \cite[Corollary 4.14]{BL}  respectively).

\begin{thm}[Accessibility of $c$]\label{Accessibility of c}
If $f_c$ is an exponential map with $c\in J(\fc)$ and $\PP(\fc)$ is bounded, then there is at least a dynamic ray $g_\s$ landing at $c$, and $\s$ is a bounded address. Also, the length of the arcs $g_{\sigma^n{\s}}(0,t)\ra0$ uniformly in $n$ as $t\ra0$.
\end{thm}
The last estimate and the fact that $\s$ is bounded  are not stated explicitly, but they follow directly from the construction. In particular, if the postsingular set is bounded, $c$ satisfies automatically the hypothesis of Theorem D hence Corollary D follows immediately. Theorem~\ref{Accessibility of c} is expected to hold also for non-recurrent parameters with unbounded postsingular set.


  
 An address or angle $\s$ is  called \emph{non-recurrent} if $\s\notin\ov{O(\s)}:=\ov{\{\sigma^n(\s)\}}_{n\in\N}$.

  
 \begin{lem}\label{Non recurrent addresses}
 Let $f_c(z)=e^z+c$ or $f_c(z)=z^D+c$ be  non-recurrent. If there is a dynamic ray  $g_\s$ landing at $c$ such that  the length of the arcs $g_{\sigma^n{\s}}(0,t)\ra0$ uniformly in $n$ as $t\ra0$, then $\s$ is non-recurrent.
  \end{lem}  
    
 \begin{proof}
 By non-recurrence there is a disk $D_\epsilon(c)$ such  that $D_\epsilon(c)\cap\PP(f_c)=\emptyset$. 
  As $g_\s$ lands at $c$, for any $k\in\N$ the dynamic ray $g_{\sigma^k\s}$ lands at  $f_c^k(c)$.
 Let $t_\epsilon$ be such that 
  
 \begin{equation}\label{star} 
 \ell(g_{\sigma^n{s}}(0,t))<\epsilon/2 \end{equation} for all $n\in\N,t<t_\epsilon$.
 Also let $t_0<t_\epsilon$ such that $g_\s(t_0)\in D_{\epsilon/2}(c)$.

 If $\s$ is recurrent, there is a subsequence ${\sigma^k\s\ra\s}$ as $k\ra\infty$ and such that $t_{\sn,c}\ra t_{\s,c}$, hence by Lemma~\ref{ContRays}, $g_{\sigma^k\s}(t_0)\ra g_\s (t_0)$ as $k\ra\infty$ ($\tsc=0$ since $c$ is non-escaping).
 By Equation~\ref{star}, $g_{\sigma^k\s}$ eventually lands inside $D_\epsilon(c)$, contradicting $D_\epsilon(c)\cap\PP(f_c)=\emptyset$.
 \end{proof}

The next lemma establishes a relation between the landing of a ray in parameter plane and the landing of the ray with the same address in dynamical plane. 
 
 \begin{lem}\label{Dynamical Parameter fiber}
 Let $f_c$ be  a unicritical polynomial or an exponential map. If a parameter ray $G_\s$ lands at a parameter  $c$ and $c\in J(\fc)$, then $g^c_\s$ belongs to the dynamic fiber of  $c$. Viceversa, if $g^c_\s$ lands at $c$, the parameter ray $G_\s$ belongs to the parameter fiber of $c$.
 \end{lem}
 
 \begin{proof}
Suppose that  $G_\s$ lands at $c$ and that  $g^c_\s$  does not belong to the dynamic fiber of  $c$. Then there is a pair of dynamic  rays $g^c_{\s^+},g^c_{\s^-}$ separating $g^c_\s$  from  $c$, whose addresses are either periodic or preperiodic. If the addresses are periodic, their forward iterates form an orbit portrait, hence there are also two characteristic rays enclosing $c$ which separate  $g^c_\s$  from  $c$, and  the corresponding parameter rays $G_{\s^+},G_{\s^-}$ form a parabolic wake in parameter plane by Proposition~\ref{Parabolic wakes}. If the addresses are preperiodic,  the parameter rays $G_{\s^+},G_{\s^-}$ form a  Misiurewicz wake by  Proposition~\ref{Misiurewicz wakes}. In both cases, the wake separates $G_\s$  from $c$ by vertical order of parameter rays in parameter plane, so $G_\s$  does not belong to the fiber of $c$.
The case in which   $g^c_\s$ lands at at $c$ and   $G_\s$  does not belong to the fiber of  $c$ is analogous.
 \end{proof}

 We now spend some time proving that for unicritical polynomials parameter rays with non-recurrent addresses land at non-recurrent parameters.
  
 \begin{prop}[Landing of non-recurrent rays]\label{Landing of non-recurrent rays}
 Let $G^D_\s$ be a parameter ray of angle $\s$ in the parameter plane for the family $\{P_c^D\}$. If $\s$ is non-recurrent, then $G^D_\s$ lands at a parameter $\ctilde$ which is at most finitely renormalizable and non-recurrent.
 \end{prop}
 
 What follows is a brief introduction to renormalization and rigidity for unicritical polynomials.  
A unicritical polynomial $P^D_c(z)=z^D+c$ with connected Julia set is called \emph{renormalizable of period $n$} if there are neighborhoods $\emph{U,U'}$ of $0$, with $U$ compactly contained in $U'$, such that $f^n:U\ra U'$ is a degree $D$ polynomial-like map with connected Julia set (see e.g. \cite{Hu}). The new polynomial-like map can be itself renormalizable, and so on.

For quadratic polynomials, the connected components of sets of parameters which are renormalizable of period $n$ form  small copies of the Mandelbrot set (see \cite{DH}). Parameters which are renormalizable of period $n$, and hence these copies, are combinatorially characterized by being contained in some parabolic wake bounded by two characteristic angles $\s^+,\s^-$, from which infinitely many Misiurewicz wakes have been cut out. We will call this the \emph{renormalization wake} of the corresponding small copy. The dynamic rays which are left in this renormalization wake are exactly the ones whose angles are represented by  sequences in $\{t^+,t^-\}^\N$, where $t^+,t^-$ are the  finite sequences of $n$ symbols with $\sp=\ov{t^+}$, $\sm=\ov{t^-}$.  This is the same as saying that the angle is the output of a \emph{tuning} through the angles $\sp,\sm$ as described in \cite{Do}.
Observe that a parameter ray accumulates on a given small copy if and only if it belongs to its renormalization wake.

For unicritical polynomials of higher degree  $D$, renormalizable parameters of period $n$ also form small copies of the  connectedness locus $M_D$ and can be characterized in a similar way (see \cite[Theorem 3.1]{Sc3}). The renormalization wake is again formed by a parabolic wake, from which countably many Misiurewicz wakes have been cut out. The angles of the parameter  rays contained in the renormalization wake can be describes as the sequences in $\{t^+,t^-, t^3\ldots t^{D-2}\}^\N$ where $t^i$ are a finite number of  specific  sequences of length $n$ (these sequences  can be characterized explicitly). Again, a parameter ray accumulates on a given small copy if and only if it belongs to its renormalization wake.


An angle is called \emph{combinatorially renormalizable} of period $n$  if it belongs to the renormalization wake of a small copy of renormalization period $n$.  
    
\begin{lem}\label{Finitely renormalizable} If $\s$ is a non-recurrent angle written in $D$-adic expansion, then it is at most finitely many times combinatorially renormalizable (when seen as the angle of a parameter ray for polynomials of degree $D$).
\end{lem}
\begin{proof} If $\s$ is renormalizable of some period $q$, then $\s$ is represented by a sequence constructed out  of only $D$ blocks of length $q$, hence there are infinitely many $k$ such that $|\s-\sigma^k\s|<\frac{2}{D^q}$. If $\s$ is infinitely many times combinatorially renormalizable, then  there exists a sequence $q_n\ra\infty$ and $k_n\ra\infty$ such that  $|\s-\sigma^{k_n}\s|<\frac{2}{D^{q_n}}$, contradicting non-recurrence of $\s$.   
 \end{proof}

 The next theorem is one of the cornerstones in the theory of rigidity for quadratic and unicritical polynomials (see \cite{Hu}, \cite{AKLS} respectively).
 
 \begin{thm}[Yoccoz Theorem]
  Let $P^D_c$ be an at most finitely renormalizable unicritical polynomial with all periodic points repelling and connected Julia set. Then $P^D_c$ is combinatorially rigid, or equivalently the reduced fiber of $c$ is a single point. 
 \end{thm}
 
 The condition of combinatorial rigidity is often stated as \emph{parapuzzle pieces shrink to points}.
To prove Yoccoz Theorem local connectivity of the Julia set is needed. For unicritical polynomials see \cite[Theorem A]{KL}:

\begin{thm}\label{LC}
  Let $P^D_c$ be an at most finitely renormalizable unicritical polynomial with all periodic points repelling and connected Julia set. Then $J(P^D_c)$ is locally connected.
 \end{thm}
 
By Caratheodory's Theorem, local connectivity of the Julia set implies that all dynamic rays land and all points are landing points of a dynamic ray. In particular, the critical value  is accessible.  
It is also known that  for a non-renormalizable unicritical polynomial  with all periodic points repelling and connected Julia set, the puzzle pieces induced by the periodic rays landing at the $\alpha$-fixed points shrink to points. So, for  an at most finitely renormalizable polynomial of the same kind, the puzzle pieces induced by some forward invariant graph $\tilde{\Gamma}$ formed by finitely  many periodic rays also shrink to points.

 After the preparation about tuning, Proposition~\ref{Landing of non-recurrent rays} is essentially a consequence of Yoccoz's Theorem.

 \begin{proof}[Proof of Proposition \ref{Landing of non-recurrent rays}]
 By Lemma~\ref{Finitely renormalizable}, $\s$ is only finitely renormalizable, so $G_{\s}$ belongs to at most finitely many renormalization wakes and can accumulate only on non-escaping parameters which are at most finitely renormalizable. By Yoccoz Theorem, the  reduced fiber of any  such parameter $\ctilde$ is trivial,   hence if  $G_\s^D$ accumulates on $\ctilde$ in fact it lands at it. By Proposition~\ref{Dynamical Parameter fiber}, $g^\ctilde_\s$ belongs to the dynamical fiber of $\ctilde$, and by local connectivity of $J(f_\ctilde)$ in fact it lands at $\ctilde$.  As $\s$ is  a non-recurrent address,  $\ctilde$ is non-recurrent.
 \end{proof}
 
 We present two more results and then prove Theorem D.

\begin{prop}[{\cite[Theorem 4.7]{Be}}]\label{Misiurewicz wakes correspondence}
Let $\s_1\ldots\s_q$ be a finite set of preperiodic addresses. Then the parameter rays of addresses $\s_1\ldots\s_q$ land together at a Misiurewicz parameter in the exponential parameter plane if and only if they land together at a Misiurewicz parameter in the  parameter plane of any family of unicritical polynomials with sufficiently high degree $D$.
\end{prop}

The next proposition is a direct consequence of {\cite[Theorem 4.11]{Be}} and Theorem~\ref{Parabolic wakes}.
\begin{prop}\label{Parabolic wakes correspondence}
Two parameter rays with periodic addresses $\s^+,\s^-$ land together in the parameter plane for exponentials if and only if the parameter rays with angles $\s^+,\s^-$ land together in the parameter plane for all families $\{P_c^D\}$ with sufficiently high degree $D$.
\end{prop}
 
 \begin{proof}[Proof of Theorem D]
 
 Let $g_\s$ be the dynamic ray landing at $c$ given by Theorem \ref{Accessibility of c} and let $D$ be sufficiently large. By Lemma~\ref{Non recurrent addresses}, $\s$ is non-recurrent, so by Proposition~\ref{Landing of non-recurrent rays} the parameter ray $G^D_\s$ (for the family of unicritical polynomials of degree $D$) lands at a finitely renormalizable non-recurrent polynomial parameter $\ctilde$.  

By Lemma~\ref{Dynamical Parameter fiber},   the dynamic ray $g^\ctilde_\s$ belongs to the dynamical fiber of $\ctilde$. As $\ctilde$ is at most finitely renormalizable, by Theorem~\ref{LC} its Julia set is locally connected, so in fact $g^\ctilde_\s$  lands at $\ctilde$. Because $\ctilde$ is at most finitely renormalizable, there is some $n$ and some cycle $\tilde{\Gamma}$ of rays of period $n$ such that the puzzle pieces induced by $\tilde{\Gamma}$ and containing $c$ shrink to points.

By non-recurrence of $\ctilde$, $\dist(\ctilde,\PP(f_{\ctilde}))>0$, hence there is some level of the  puzzle induced by $\tilde{\Gamma}$ for which the singular value and $\PP(f_\ctilde)$ are separated. This means that there are finitely many preperiodic rays $\{g^\ctilde_{\s_i}\}_{\s_i\in\KK}$ separating $\ctilde$ from $\PP(f_{\ctilde})$.

By the polynomial analogue of Propositions~\ref{Parabolic wakes} and   \ref{Misiurewicz wakes}, $\ctilde$ is contained in  finitely many  wakes defined by the parameter rays $\{G^D_{\s_i}\}_{\s_i\in\KK}$ (it can be contained in many more wakes, but this is irrelevant to us).

By Propositions~\ref{Misiurewicz wakes correspondence} and \ref{Parabolic wakes correspondence}, the same wakes exist in the exponential parameter plane,  and by vertical order of parameter rays, the parameter ray  $G_\s$  (and hence $c$ by Lemma~\ref{Dynamical Parameter fiber}) is contained in all of them. By Proposition~\ref{Parabolic wakes} and \ref{Misiurewicz wakes}, the corresponding  pairs of  rays with addresses in $\KK$ land together in the dynamical plane for $c$. By vertical order of dynamical rays, as the rays $g^\ctilde_{\s_i}$ separate  $\ctilde$ from its forward orbit in the polynomial dynamical plane (and each point on the orbit is the landing point of a dynamic ray $g^\ctilde_{\sigma^k \s}$ by Theorem~\ref{Accessibility of c}), the rays $\{g^c_{\s_i}\}_{\s_i\in\KK}$ also separate  $c$ from its forward orbit in the exponential plane. By continuity, they also separate $c$ from the closure of its forward orbit i.e. from the postsingular set. 
 \end{proof}

\begin{small}

\end{small}

\end{document}